\documentclass{article}

\usepackage[verbose]{geometry}
\usepackage[T1]{fontenc}

\usepackage{authblk}

\usepackage{amsmath, amssymb, amsthm}
\usepackage{mathtools}
\usepackage{mathrsfs}
\usepackage{abstract}
\newcommand{\Cov}{\operatorname{Cov}}
\newcommand{\Tr}{\operatorname{Tr}}
\newcommand{\Lip}{\operatorname{Lip}}

\DeclareMathOperator*{\argmin}{arg\,min}
 
\usepackage{hyperref}
\usepackage[dvipsnames]{xcolor}
\hypersetup{
colorlinks=true,
allcolors=MidnightBlue,
linktoc=page
}

\usepackage{natbib}

\usepackage{graphicx}

\usepackage[ruled,vlined]{algorithm2e}

\theoremstyle{plain}
\newtheorem{thm}{Theorem}[section]      %
\newtheorem{lem}[thm]{Lemma}            %
\newtheorem{cor}[thm]{Corollary}
\newtheorem{prop}[thm]{Proposition}

\theoremstyle{definition}
\newtheorem{defn}[thm]{Definition}
\newtheorem{assum}[thm]{Assumption}

\theoremstyle{remark}
\newtheorem{rem}[thm]{Remark}

\DeclareMathOperator{\Law}{Law}
\DeclareMathOperator{\Var}{Var}

\newcommand{\cw}{\mathcal{W}}

\newcommand{\keywords}[1]{%
  \par\noindent\textbf{Keywords: }#1\par
}

\usepackage{changes}
\usepackage{CJKutf8}

\title{Self-fictitious-play for Potential Monotone Ergodic Mean-field Games}

\date{}

\author[1]{Yupeng Bai}
\author[2]{Mathieu Laurière}
\author[1]{Zhenjie Ren}
\author[3]{Songbo Wang}

\affil[1]{LaMME, Université Evry Paris-Saclay, France}
\affil[2]{NYU Shanghai, China}
\affil[3]{Laboratoire J.A.Dieudonné,
Université Côte d'Azur, France}

\begin{document}

\maketitle

\begin{abstract}
	We investigate long-time learning in ergodic, potential, monotone mean-field games (MFGs) via a self-fictitious-play (SFP) dynamics coupling an optimally controlled diffusion with a slowly evolving belief. At each time, the state follows the optimal feedback associated with the current belief, while the belief is updated using the player’s own empirical occupation measure rather than the population distribution. For ergodic monotone potential  MFGs on the torus, we prove that the SFP dynamics is contractive and admits a unique invariant law. Moreover, we show that this invariant law is quantitatively close to the MFG Nash equilibrium, with an error of order equal to the square root of the belief-update rate.
	The proof combines uniform-in-time regularity estimates for the ergodic Hamilton–Jacobi–Bellman equation with an energy argument based on the Lasry–Lions divergence. The linear-quadratic example shows that this rate is sharp, and the numerical experiments illustrate the predicted scaling.
\end{abstract}

\medskip

\section{Introduction}

Mean field games (MFGs) provide a tractable description of Nash equilibria in stochastic differential games with a very large number of weakly interacting and statistically indistinguishable agents. Initiated independently by Lasry--Lions~\cite{LasryLions2006I,LasryLions2006II,LasryLions2007} and by Huang--Malham\'e--Caines~\cite{HuangMalhameCaines2006}, the theory identifies equilibria through a consistency condition: given a population distribution, a representative agent solves an optimal control problem; the induced distribution of optimally controlled states must coincide with the original one. In the diffusion setting, this fixed-point structure is typically encoded by a coupled Hamilton--Jacobi--Bellman (HJB) and Fokker--Planck (FP) system. Central themes in the literature are to understand existence, uniqueness, and stability of equilibria, as well as to design algorithms that approximate them.

This paper proposes a probabilistic approximation scheme for the Nash equilibrium of \emph{potential, monotone, ergodic} MFGs, based on a self-interacting diffusion. Our approach can be viewed as a continuous-time learning dynamics: the agent repeatedly best-responds to an evolving belief on the population, while the belief itself is updated through an occupation measure of her own single trajectory. We prove quantitative asymptotic guarantees for the resulting dynamics and derive an explicit small-parameter error bound in Wasserstein distance.

\medskip
\noindent\textbf{From finite-horizon control to the ergodic HJB equation.}
We begin with the finite-horizon stochastic control problem on $E$ (with $E$ being the torus $\mathbb{T}^d$ or the whole space $\mathbb{R}^d$) driven by a standard Brownian motion $W_t$:
\[
	u^{T}(0,x)
	= \inf_{\alpha}
	\mathbb{E}\left[
		\int_{0}^{T} c(X_t,\alpha_t)\,dt + g(X_T)
		\right],
	\qquad
	X_0 = x,\qquad dX_t = \alpha_t\,dt + \sqrt{2}\, dW_t .
\]
The control $\{\alpha_t\}_{0\le t\le T}$ is progressively measurable and aims to minimize the running cost $\int_0^T c(X_t,\alpha_t)\,dt$ plus the terminal cost $g(X_T)$. The associated value function is
\[
	u^T(t,x)
	=
	\inf_{\alpha}
	\mathbb{E}\left[
		\left.
		\int_{t}^{T} c(X_s,\alpha_s)\,ds + g(X_T)
		\right| X_t=x
		\right].
\]
By the dynamic programming principle, $u^T$ solves the HJB equation (see, e.g.,~\cite{Peng1992SHJB,Bellman1958})
\[
	\partial_t u^T(t,x) + \Delta u^T(t,x) + H\bigl(x,\nabla u^T(t,x)\bigr) = 0,
	\qquad
	u^T(T,x)=g(x),
\]
where the Hamiltonian is defined by
\[
	H(x,p) = \inf_{\alpha} \bigl( \alpha\cdot p + c(x,\alpha) \bigr).
\]

In many applications the relevant performance criterion is the long-run average cost. This leads to the ergodic control problem, which studies the behavior of optimal costs as $T\to\infty$. Formally, define (when the limit exists)
\[
	\bar \delta
	:=
	\lim_{T\to\infty}
	\frac{1}{T}
	\inf_{\alpha}
	\mathbb{E}\left[
		\int_0^T c(X_t,\alpha_t)\,dt + g(X_T)
		\right],
\]
so that $\bar\delta$ represents the asymptotic average cost rate. For large horizons, optimal trajectories often exhibit a turnpike behavior: after an initial transient, the controlled dynamics remain close to a steady regime for most of the time interval before a final adjustment near $T$; see~\cite{McKenzie1976Turnpike,GeshkovskiZuazua2022}. Under suitable assumptions, this implies
\begin{align*}
	\frac{u^T(0,x)}{T}       & \longrightarrow \bar \delta,
	\qquad \text{as } T\to\infty,                           \\
	u^T(0,x) - \bar \delta T & \longrightarrow u(x),
	\qquad \text{as } T\to\infty,
\end{align*}
where $u$ is the stationary corrector. The ergodic limit is then characterized by an ergodic HJB equation coupled with the invariant measure $\bar m^*$ of the optimally controlled dynamics:
\[
	\begin{cases}
		-\bar \delta + \Delta u(x) + H\bigl(x,\nabla u(x)\bigr) = 0, \\
		\Delta \bar m^*(x) - \nabla\cdot\bigl(\bar \alpha^*(x)\,\bar m^*(x)\bigr) = 0 \quad\text{in }E.
	\end{cases}
\]

\medskip
\noindent\textbf{Ergodic MFGs.}
An ergodic MFG describes the stationary behavior of a large population of agents whose running costs depend on the population distribution. Given a distribution $m$, let $f(m,x)$ denote the interaction cost faced by an agent at state $x$. The representative player considers the ergodic control problem
\[
	\delta
	=
	\lim_{T\to\infty}
	\frac{1}{T}
	\inf_{\alpha}
	\mathbb{E}\left[
		\int_0^T
		\bigl(
		c(X_t,\alpha_t) + f(m, X_t)
		\bigr)\,dt
		+ g(X_T)
		\right].
\]
Since $f$ does not depend on $\alpha$, the optimal feedback again derives from the Hamiltonian, and the controlled diffusion takes the form
\[
	dX_t =  H_p\bigl(X_t,\,\nabla u(X_t)\bigr)\,dt + \sqrt{2}\,dW_t,
\]
where $u$ solves the ergodic HJB equation associated with $m$. Assuming this diffusion admits a unique invariant measure, denoted $\widehat m$, the map
\[
	m\ \longmapsto\ \widehat m
\]
is the best-response operator of the game. An ergodic Nash equilibrium is a fixed point $m^*=\widehat m$: given $m^*$, each agent plays an optimal ergodic control, and the induced stationary distribution is again $m^*$. At the PDE level, $(u^*,m^*,\delta^*)$ solve the ergodic HJB--FP system
\begin{equation}\label{eq-HJB-FP-intro}
	\begin{cases}
		-\delta^* + \Delta u^*(x) + H(x, \nabla u^*(x)) + f(m^*,x) = 0, \\
		\Delta m^*  - \nabla \cdot \,\left( m^* \, H_p\big(x, \nabla u^*(x)\big) \right) = 0.
	\end{cases}
\end{equation}

The long-time behavior of finite-horizon MFG systems and their convergence to the ergodic regime have been studied extensively. Under the Lasry--Lions monotonicity condition, exponential turnpike results for solutions of the ergodic HJB--FP system were established in~\cite{CardaliaguetLasryLionsPorretta2012,CardaliaguetDelarueLasryLions2019}. More recently,~\cite{CecchinConfortiDurmusEichinger2024} proves a turnpike property under a ``weak interaction/strong ergodicity'' regime. We will use estimates from~\cite{CecchinConfortiDurmusEichinger2024} to obtain the regularity bounds needed for our analysis of the ergodic HJB--FP system.

\medskip
\noindent\textbf{Potential games, mean field control, and a variational perspective.}
An ergodic MFG is called \emph{potential} if the interaction term $f(m,x)$ is the linear functional derivative (see Definition~\ref{def:linear derivative}) of a potential functional $F$, namely
\[
	f(m,x)=\frac{\delta F}{\delta m}(m,x).
\]
In this case, under mild assumptions, the equilibrium control $\alpha^*$ also solves the corresponding mean field control (MFC) problem
\[
	\lim_{T\to\infty}\frac{1}{T}\inf\left\{
	\mathbb{E}\!\left[\int_0^T\Bigl(c(X_t,\alpha_t)+F(m_t)\Bigr)\,dt\right]
	:\ dX_t=\alpha_t\,dt+\sqrt{2}\,dW_t,\ \ m_t=\mathrm{Law}(X_t)
	\right\}.
\]
A rigorous probabilistic formulation of MFC was developed by Carmona and Delarue~\cite{CarmonaDelarue2015}, and the relation between MFG and MFC is discussed for instance in~\cite{CarmonaDelarueLachapelle2013}. In the special case of quadratic running cost $c(x,\alpha)=\tfrac12|\alpha|^2$, one has $\alpha^*=-\nabla u^*$ and the equilibrium density satisfies $m^*\propto e^{-u^*}$, hence $\alpha^*=\nabla\log m^*$. This identifies $m^*$ as a natural candidate for the static variational problem
\[
	\inf_m \left\{\int \frac12\,|\nabla \log m(x)|^2\, m(dx) + F(m)\right\},
\]
which can be interpreted as mean-field optimization regularized by the Fisher information, as studied in~\cite{ClaisseConfortiRenWang2023}.

\medskip
\noindent\textbf{Dynamic approximation of ergodic equilibria: from McKean--Vlasov learning to self-interaction.}
To approximate the ergodic Nash equilibrium, we study a continuous-time learning dynamics. Before introducing it, we recall the approach of Mouzouni~\cite{Mouzouni2020QuasiStationaryMFG}, who proposed an ergodic analogue of fictitious play based on a McKean--Vlasov diffusion. Let $(m_t)_{t\ge 0}$ be the population flow and consider
\[
	dX_t
	=
	H_p\left( X_t,\, \nabla u^{m_t}(X_t) \right) dt
	+ \sqrt{2}\, dW_t,
	\qquad
	m_t = \mathcal{L}(X_t),
\]
where $u^m$ solves, for fixed $m$,
\[
	-\delta^m
	+ \Delta u^m(x)
	+ H\bigl(x, \nabla u^m(x)\bigr)
	+ f(m,x)
	= 0.
\]
This evolution couples a best response to the current belief $m_t$ with the consistency update $m_t=\Law(X_t)$. Under a smallness condition on the initial discrepancy (in $L^2$) between $m_0$ and the equilibrium,~\cite{Mouzouni2020QuasiStationaryMFG} proves convergence $m_t\to m^*$ as $t\to\infty$. This construction is closely related in spirit to fictitious play schemes for finite-horizon MFGs, studied for instance in~\cite{CardaliaguetHadikhanloo2017,PerrinFPMFG}.

A different line of work shows that replacing the interaction through the current marginal $\Law(X_t)$ by an interaction through an \emph{occupation measure} can be both natural and effective in mean-field sampling and optimization. The long-time behavior of self-interacting diffusions was developed in foundational works such as~\cite{CranstonLeJan1995,Raimond1997} and in the series of papers by Bena\"im and coauthors~\cite{BenaimLedouxRaimond2002,BenaimRaimond2003,BenaimRaimond2005,BenaimRaimond2011}. Recently, a quantitative analysis of this idea was proposed in~\cite{DuRenSuciuWang2024} for approximating the invariant measure of the mean-field Langevin diffusion on $\mathbb{R}^d$,
\[
	dX_t = - D_m F({\rm Law}(X_t), X_t)\, dt + \sqrt{2}\, dW_t,
	\qquad m_t = \mathcal{L}(X_t),
\]
whose unique invariant measure $m_{+\infty}$ exists under convexity assumptions on $F$ (see~\cite{HRSS}). In~\cite{DuRenSuciuWang2024}, the authors consider instead the self-interacting dynamics
\begin{align*}
	dX_t & = - D_m F(m_t, X_t)\, dt + \sqrt{2}\, dW_t,      \\
	dm_t & = \lambda \bigl( \delta_{X_t} - m_t \bigr)\, dt,
\end{align*}
and prove quantitative bounds comparing the induced invariant measure $\rho^*$ with $m_\infty$, with error of order $O(\sqrt{\lambda})$ in Wasserstein distance.

\medskip
\noindent\textbf{Our contribution: Self-Fictitious-Play for ergodic MFGs.}
Motivated by fictitious play for MFGs and by quantitative results for self-interacting diffusions, we propose a self-interacting approximation scheme for ergodic MFG equilibria that we call \emph{Self-Fictitious-Play} (SFP). The state process follows
\begin{equation}\label{eq:X intro}
	dX_t
	=
	H_p\left( X_t,\, \nabla u^{m_t}(X_t) \right)\, dt
	+ \sqrt{2}\, dW_t,
\end{equation}
but, in contrast with the McKean--Vlasov update $m_t=\Law(X_t)$, we update the belief through the continuous-time SFP rule
\[
	dm_t = \lambda \bigl( \delta_{X_t} - m_t \bigr)\, dt.
\]
Thus $m_t$ is a weighted empirical measure of the trajectory, and the pair $(X_t,m_t)$ forms a self-interacting diffusion. Our main results show that, for monotone potential ergodic MFGs on $\mathbb{T}^d$, the distance (in Wasserstein sense) between the Nash equilibrium $m^*$ and the marginal distribution induced by SFP converges, as $t\to\infty$, to a neighborhood of size $O(\sqrt\lambda)$. Analogous estimates can be obtained on $\mathbb{R}^d$ under suitable conditions; however, to keep a uniform and streamlined presentation we focus on the torus throughout the paper. The linear-quadratic example in Section~\ref{subsection-ex1} shows that the order $O(\sqrt{\lambda})$ is sharp in general. 

A key difficulty is that $u^m$ is defined through an ergodic PDE and the drift in~\eqref{eq:X intro} is generally not of gradient type. Consequently, entropy-based arguments developed for gradient systems (such as in~\cite{DuRenSuciuWang2024}) are not  applicable. Instead, we combine reflection coupling with the Lasry--Lions divergence (see Section~\ref{subsection-pre3}) to derive Wasserstein contractivity and to obtain quantitative error bounds.

Related trajectory-based mean-field learning ideas also appear in model-free reinforcement learning: finite-state Q-learning schemes have been proposed in which the population distribution is estimated from the trajectory of a representative player while the value function is updated on a separate timescale~\cite{angiuli2022unified}, with convergence of related multi-timescale algorithms studied in~\cite{angiuli2023convergence}.

\medskip
\noindent\textbf{Organization of the paper.}
Section~\ref{section-preliminary} introduces the main definitions and the standing assumptions. Section~\ref{section-main-result} contains our main results: Subsection~\ref{subsection-contractivity} proves convergence of the Self-Fictitious-Play dynamics under suitable regularity conditions, while Subsection~\ref{subsection-distance} quantifies the distance between the stationary distribution of SFP and the ergodic MFG Nash equilibrium, showing that this distance is of order $O(\sqrt{\lambda})$. Section~\ref{section-examples} presents illustrative examples: Subsection~\ref{subsection-ex1} studies a linear-quadratic case suggesting the size $O(\sqrt{\lambda})$ is sharp, and Subsection~\ref{subsection-ex2} provides numerical illustrations in the LQ case and in another model with explicit solution. Proofs are collected in Section~\ref{section-proof}.

\section{Preliminaries}\label{section-preliminary}

\subsection{Ergodic monotone potential MFGs}\label{subsec:pre:mfg}
Let $(\mathbb{T}^d,\mathcal{F},(\mathcal{F}_t)_{t\ge0},\mathbb{P})$ be a filtered probability space satisfying the usual conditions, and let $(W_t)_{t\ge0}$ be a standard $d$-dimensional Brownian motion. We fix a running cost
\[
	c:\mathbb{T}^d\times\mathbb{R}^d\to\mathbb{R}
	\qquad\text{and}\qquad
	f:\mathcal{P}(\mathbb{T}^d)\times\mathbb{T}^d\to\mathbb{R},
\]
where $f(m,x)$ models the interaction between an individual at state $x$ and the population distribution $m$.

\medskip
\noindent\textbf{Ergodic control with frozen population.}
For a fixed $m\in\mathcal{P}(\mathbb{T}^d)$, consider the controlled diffusion
\[
	dX_t=\alpha_t\,dt+\sqrt{2}\,dW_t,
\]
and define the corresponding ergodic cost (when it exists) by
\begin{equation}\label{eq:def-delta-m}
	\delta^m
	:=
	\lim_{T\to\infty}\frac1T
	\inf_{\alpha}
	\mathbb{E}\Bigl[\int_0^T\bigl(c(X_t,\alpha_t)+f(m,X_t)\bigr)\,dt\Bigr].
\end{equation}
The Hamiltonian is
\[
	H(x,p)=\inf_{\alpha\in\mathbb{R}^d}\bigl(\alpha\cdot p + c(x,\alpha)\bigr),
	\qquad
	H_p(x,p)=\partial_p H(x,p),
\]
and the (formal) optimal feedback is $\alpha^*(x)=H_p(x,\nabla u(x))$ where $u$ solves the associated ergodic HJB equation.

\medskip
\noindent\textbf{Ergodic MFG equilibrium.}
An ergodic mean field Nash equilibrium is a triple $(\delta^*,u^*,m^*)$ such that $m^*$ is an invariant measure of the optimally controlled dynamics corresponding to the population $m^*$ itself. At the PDE level, this is encoded by the stationary HJB--FP system
\begin{equation}\label{eq-MFG}
	\begin{cases}
		-\delta^* + \Delta u^*(x) + H(x,\nabla u^*(x)) + f(m^*,x) = 0,
		 & x\in\mathbb{T}^d, \\[2mm]
		\Delta m^*(x) - \nabla\!\cdot\!\Bigl(m^*(x)\,H_p\bigl(x,\nabla u^*(x)\bigr)\Bigr)=0,
		 & x\in\mathbb{T}^d, \\[1mm]
		\int_{\mathbb{T}^d} m^*(dx)=1.
	\end{cases}
\end{equation}

\medskip
\noindent\textbf{Potential structure and monotonicity.}
We recall the notion of (linear) functional derivative.

\begin{defn}[Linear derivative]\label{def:linear derivative}
	Let $F:\mathcal{P}(\mathbb{T}^d)\to\mathbb{R}$. A function
	\[
		\frac{\delta F}{\delta m}:\mathcal{P}(\mathbb{T}^d)\times\mathbb{T}^d\to\mathbb{R}
	\]
	is called a \emph{linear derivative} of $F$ if for all $m,m'\in\mathcal{P}(\mathbb{T}^d)$,
	\[
		F(m')-F(m)
		=\int_0^1\int_{\mathbb{T}^d}
		\frac{\delta F}{\delta m}(m^s,x)\,(m'-m)(dx)\,ds,
		\qquad m^s:=(1-s)m+s m'.
	\]
\end{defn}

\begin{defn}[Monotone potential ergodic MFG]\label{def:mono-pot}
	We consider ergodic MFGs satisfying:
	\begin{enumerate}
		\item \textbf{Potentiality:} there exists a functional $F$ such that
		      \[
			      f(m,x)=\frac{\delta F}{\delta m}(m,x).
		      \]
		\item \textbf{Lasry--Lions monotonicity:} for all $m,m'\in\mathcal{P}(\mathbb{T}^d)$,
		      \[
			      \int_{\mathbb{T}^d}\bigl(f(m,x)-f(m',x)\bigr)\,(m-m')(dx)\ge0.
		      \]
	\end{enumerate}
	In the potential case, monotonicity is equivalent to convexity of $F$ along affine interpolations (often called \emph{flat convexity}).
\end{defn}

Throughout the paper we focus on monotone potential ergodic MFGs on $\mathbb{T}^d$.

\subsection{Cylindrical potentials}\label{subsec:pre:cyl}
We specialize to cylindrical potentials, which allow us to reduce the self-interacting system to finite dimension.

\begin{assum}[Regularity of the potential]\label{ass-potential}
	Assume that $F$ is of the form
	\[
		F(m)=\Phi(\langle \ell,m\rangle)
		:=\Phi\Bigl(\int_{\mathbb{T}^d}\ell(x)\,m(dx)\Bigr),
	\]
	where $\ell=(\ell_1,\dots,\ell_D)\in C^2(\mathbb{T}^d;\mathbb{R}^D)$ and $\Phi\in C^2(\mathbb{R}^D;\mathbb{R})$ is convex. Then
	\[
		f(m,x)=\frac{\delta F}{\delta m}(m,x)
		=\nabla\Phi(\langle \ell,m\rangle)\cdot \ell(x).
	\]
	We quantify regularity by constants $C_\ell,C_\Phi>0$ such that:
	\begin{enumerate}
		\item[(i)] for all $x\in\mathbb{T}^d$,
			\[
				\|\ell(x)\| + \|\nabla\ell(x)\| + \|\nabla^2\ell(x)\|\le C_\ell;
			\]
		\item[(ii)] for all $\|y\|\le C_\ell$,
			\[
				\|\nabla\Phi(y)\| + \|\nabla^2\Phi(y)\|\le C_\Phi .
			\]
	\end{enumerate}
\end{assum}

\subsection{Self-fictitious play and finite-dimensional reduction}\label{subsec:pre:sfp}
Fix $\lambda>0$. The \emph{self-fictitious play} (SFP) dynamics are defined by
\begin{equation}\label{eq-SFP}
	\begin{cases}
		dX_t = H_p\bigl(X_t,\nabla u^{m_t}(X_t)\bigr)\,dt + \sqrt{2}\,dW_t, \\
		dm_t = \lambda\bigl(\delta_{X_t}-m_t\bigr)\,dt,
	\end{cases}
\end{equation}
where, for each $m\in\mathcal{P}(\mathbb{T}^d)$, the function $u^m$ solves the ergodic HJB equation
\[
	-\delta^m + \Delta u^m(x) + H\bigl(x,\nabla u^m(x)\bigr) + f(m,x)=0.
\]

Under Assumption~\ref{ass-potential}, the interaction depends on $m$ only through the moment
\[
	y:=\langle \ell,m\rangle\in\mathbb{R}^D.
\]
Define
\[
	V(x,y):=\nabla\Phi(y)\cdot\ell(x),
	\qquad\text{so that}\qquad
	f(m,x)=V\bigl(x,\langle \ell,m\rangle\bigr).
\]
Setting $Y_t:=\langle \ell,m_t\rangle$, the second equation in~\eqref{eq-SFP} yields the closed finite-dimensional dynamics
\[
	dY_t=\lambda\bigl(\ell(X_t)-Y_t\bigr)\,dt.
\]
Consequently, the SFP system can be equivalently written on $\mathbb{T}^d\times\mathbb{R}^D$ as
\begin{equation}\label{eq-SFP-XY}
	\begin{cases}
		dX_t = H_p\bigl(X_t,\nabla u^{Y_t}(X_t)\bigr)\,dt + \sqrt{2}\,dW_t, \\
		dY_t = \lambda\bigl(\ell(X_t)-Y_t\bigr)\,dt,
	\end{cases}
\end{equation}
where, by abuse of notation, $u^y$ denotes the solution of
\begin{equation}\label{eq-HJB-xy}
	-\delta^y + \Delta u^y(x) + H\bigl(x,\nabla u^y(x)\bigr) + V(x,y)=0.
\end{equation}
As usual in ergodic HJB theory, $u^y$ is understood up to an additive constant; we fix a normalization throughout the paper.

\subsection{Lasry--Lions divergence}\label{subsection-pre3}
A key auxiliary quantity in our analysis is a quantity introduced by Lasry and Lions~\cite{LasryLions2007} to establish uniqueness and stability in monotone MFGs; we will refer to it as the Lasry--Lions divergence. Here we use it to quantify the discrepancy between the equilibrium feedback and the feedback induced by the $Y$-component of SFP.

\begin{defn}[Lasry--Lions divergence]\label{def:LL}
	Let $u,u'\in C^2(\mathbb{T}^d)$ and let $m$ (resp.\ $m'$) be the invariant measure of the diffusion
	\[
		dX_t = H_p\bigl(X_t,\nabla u(X_t)\bigr)\,dt + \sqrt{2}\,dW_t
		\qquad
		(\text{resp. }u').
	\]
	The \emph{Lasry--Lions divergence} between $\nabla u$ and $\nabla u'$ is
	\[
		\begin{aligned}
			D_{\mathrm{LL}}(\nabla u,\nabla u')
			 & :=\int_{\mathbb{T}^d}\Bigl(
			H\bigl(x,\nabla u(x)\bigr)-H\bigl(x,\nabla u'(x)\bigr)
			- H_p\bigl(x,\nabla u(x)\bigr)\cdot\nabla\bigl(u-u'\bigr)(x)
			\Bigr)\,m(dx)                      \\
			 & \quad+\int_{\mathbb{T}^d}\Bigl(
			H\bigl(x,\nabla u'(x)\bigr)-H\bigl(x,\nabla u(x)\bigr)
			- H_p\bigl(x,\nabla u'(x)\bigr)\cdot\nabla\bigl(u'-u\bigr)(x)
			\Bigr)\,m'(dx).
		\end{aligned}
	\]
\end{defn}

\begin{rem}
	Under standard convexity assumptions on $p\mapsto H(x,p)$, one has $D_{\mathrm{LL}}(\nabla u,\nabla u')\ge0$, and $D_{\mathrm{LL}}(\nabla u,\nabla u')=0$ typically forces the corresponding feedbacks to coincide $m$-a.e. This is the sense in which $D_{\mathrm{LL}}$ behaves like a Bregman divergence adapted to the Hamiltonian structure.
\end{rem}

\section{Main results}\label{section-main-result}

\subsection{Contractivity and existence of an invariant measure}\label{subsection-contractivity}
We next state conditions on the running cost $c$ ensuring sufficient regularity of the ergodic correctors $u^y$, and uniform Lipschitz estimates for the drift in~\eqref{eq-SFP-XY}.

\begin{assum}[Regularity and convexity of $c$]\label{ass-cost}
	Assume that $c\in C^2(\mathbb{T}^d\times\mathbb{R}^d)$ and that:
	\begin{enumerate}
		\item[\textup{(i)}] \textbf{Uniform Lipschitz in $x$:}
			\[
				\sup_{\alpha\in\mathbb{R}^d}\|c(\cdot,\alpha)\|_{\mathrm{Lip}}\le \kappa_1^c.
			\]
		\item[\textup{(ii)}] \textbf{Uniform strong convexity in $\alpha$:} there exists $\kappa_2^c>0$ such that for all $(x,\alpha)$,
			\[
				\frac{1}{\kappa_2^c} I \le c_{\alpha\alpha}(x,\alpha)\le \kappa_2^c I .
			\]
		\item[\textup{(iii)}] \textbf{Controlled growth of mixed derivatives:} there exist $\theta_c\in(0,1)$ and $\kappa_3^c>0$ such that for all $(x,\alpha)$,
			\[
				|c_{x\alpha}(x,\alpha)| \le \kappa_3^c(1+|\alpha|)^{\theta_c},
				\qquad
				|c_{xx}(x,\alpha)| \le \kappa_3^c(1+|\alpha|)^{1+\theta_c}.
			\]
	\end{enumerate}
\end{assum}

Under Assumption \ref{ass-potential} and \ref{ass-cost} on   $f$  and $c$, system~\eqref{eq-MFG} admits a unique classical solution $(\delta^*,u^*,m^*)$; see, e.g.,~\cite{LasryLions2007,GomesMitakeVoskanyan2016,Cardaliaguet2013}. We will not revisit the existence/uniqueness theory here.

\begin{rem}\label{rem:cost}
	\begin{enumerate}
		\item Since $\mathbb{T}^d$ is compact and $c(\cdot,\alpha)$ is Lipschitz uniformly in $\alpha$, $c$ is bounded in $x$ uniformly in $\alpha$. We denote by $\kappa_0^c$ a constant such that
		      \[
			      \sup_{x\in\mathbb{T}^d,\ \alpha\in\mathbb{R}^d}|c(x,\alpha)|\le \kappa_0^c.
		      \]
		\item Assumption~\ref{ass-cost}\textup{(iii)} is technical and is only used to obtain Lipschitz bounds on $\nabla u^T$ uniformly in the horizon (see Lemma~\ref{lem-gradiantu-bound}). Any alternative argument yielding the same uniform estimates would allow one to remove this condition.
	\end{enumerate}
\end{rem}

\medskip
\noindent\textbf{Uniform drift bounds.}
The next lemma records the uniform boundedness and Lipschitz continuity of the SFP drift, which is the input for our coupling estimates.

\begin{lem}[Uniform bounds for the drift]\label{lem-regularity}
	Let Assumptions~\ref{ass-potential} and~\ref{ass-cost} hold. For each $y\in\mathbb{R}^D$, let $u^y$ solve~\eqref{eq-HJB-xy}. Then there exist constants $C_b,C_L>0$ such that
	\begin{align*}
		\sup_{x\in\mathbb{T}^d,\ y\in \ell(\mathbb{T}^d)}\bigl\|H_p\bigl(x,\nabla u^y(x)\bigr)\bigr\| & \le C_b,                              \\
		\bigl\|H_p\bigl(x,\nabla u^y(x)\bigr)-H_p\bigl(x',\nabla u^{y'}(x')\bigr)\bigr\|
		                                                                                              & \le C_L\bigl(\|x-x'\|+\|y-y'\|\bigr),
	\end{align*}
	where $C_b$ and $C_L$ depend  on $(d,\kappa_1^c,\kappa_2^c,\kappa_3^c,\theta_c,C_\ell,C_\Phi)$.
	In particular, in~\eqref{eq-SFP-XY} the $X$-drift is uniformly bounded and Lipschitz, with constants independent of the realized value of $Y_t$.
\end{lem}

\medskip
\noindent\textbf{Exponential contractivity and uniqueness of the invariant law.}
We now state the contractivity estimate that yields ergodicity of SFP.

\begin{thm}[Contractivity of SFP]\label{thm-SFP-converge}
	Let Assumptions~\ref{ass-potential} and~\ref{ass-cost} hold, and fix $\lambda>0$.
	Let $(X_t,Y_t)$ and $(X'_t,Y'_t)$ be two solutions of~\eqref{eq-SFP-XY} (possibly on the same probability space) with initial laws $\Law(X_0,Y_0)$ and $\Law(X'_0,Y'_0)$. Define the weighted distance on $\mathbb{T}^d\times\mathbb{R}^D$ by
	\[
		d_\lambda\bigl((x,y),(x',y')\bigr)
		:=\|x-x'\| + \frac{2C_L}{\lambda}\|y-y'\|,
	\]
	and let $W_{d_\lambda}$ be the corresponding Wasserstein--$1$ distance on $\mathcal{P}_1(\mathbb{T}^d\times\mathbb{R}^D)$.
	Then there exist constants $C,c>0$ such that for all $t\ge0$,
	\[
		W_{d_\lambda}\bigl(\Law(X_t,Y_t),\Law(X'_t,Y'_t)\bigr)
		\le C e^{-ct}\,
		W_{d_\lambda}\bigl(\Law(X_0,Y_0),\Law(X'_0,Y'_0)\bigr).
	\]
	More precisely, with $\lambda_0:=\min\{\lambda/2,1\}$, one can take
	\[
		C = 1 +
		\frac{2M}{\sqrt{\lambda_0}}
		\exp\!\left( \frac{M^2}{4\lambda_0} \right),
		\qquad
		c =
		\left(
		\frac{1}{\lambda_0}
		+
		\frac{2M}{\lambda_0^{3/2}}
		\exp\!\left( \frac{M^2}{4\lambda_0} \right)
		\right)^{-1},
	\]
	where $M$ depends only on $(d,\kappa_1^c,\kappa_2^c,\kappa_3^c,\theta_c,C_\ell,C_\Phi,d)$.
	Consequently, the SFP Markov process~\eqref{eq-SFP-XY} admits a unique invariant probability measure $\rho_\lambda$.
\end{thm}

The argument follows the coupling approach of \cite[Theorem~1]{DuRenSuciuWang2024} (see also \cite{Eberle2016} for the underlying contraction mechanism). The details of the proof are given in Appendix~\ref{app-thm}.

\begin{rem}\label{rem:slow}
	The decay rate $c$ deteriorates as $\lambda\to0$, reflecting the slow relaxation of the $Y$-component when the update step is small. This theorem is primarily used to guarantee existence and uniqueness of $\rho_\lambda$.
\end{rem}

\subsection{Dynamic distance to the Nash equilibrium}\label{subsection-distance}

In this subsection we quantify the asymptotic bias induced by the self-updating rule in SFP.
Throughout, Assumptions~\ref{ass-potential} and~\ref{ass-cost} are in force.

\paragraph{Frozen dynamics and conditional equilibria.}
For each fixed parameter $y\in\mathbb{R}^D$, consider the \emph{frozen} diffusion
\begin{equation}\label{eq-frozen}
	dX_t^{y}
	=
	H_p\!\bigl(X_t^{y}, \nabla u^{y}(X_t^{y})\bigr)\,dt + \sqrt{2}\,dW_t,
\end{equation}
where $u^{y}$ solves the ergodic HJB equation \eqref{eq-HJB-xy}.
Since the drift in \eqref{eq-frozen} is uniformly bounded and Lipschitz in $x$ (Lemma~\ref{lem-regularity}), the process admits a unique invariant measure, denoted by $\widehat\rho_y$.

Let $\rho_\lambda$ be the unique invariant measure of the SFP system \eqref{eq-SFP-XY}, and write $\rho_\lambda^Y$ for its $Y$-marginal.
Recall that $m^*$ denotes the Nash equilibrium of the ergodic MFG, i.e.\ the $x$-component of the unique solution $(\delta^*,u^*,m^*)$ to \eqref{eq-MFG}.

\begin{thm}\label{thm-main}
	Let Assumptions~\ref{ass-potential} and~\ref{ass-cost} hold. Then:

	\begin{enumerate}

		\item (Bias of the conditional equilibrium.) There exists a constant $C_{\mathrm{fr}}>0$, depending only on $(C_b,C_L,C_\ell,d)$, such that
		      \begin{equation}\label{eq:main-fr}
			      \mathcal{W}_1\!\left(\rho_\lambda,\ \widehat\rho_{\cdot}\otimes\rho_\lambda^{Y}\right)
			      \;\le\;
			      C_{\mathrm{fr}}\,\lambda .
		      \end{equation}

		\item (Control of Lasry--Lions divergence.) The averaged Lasry--Lions divergence satisfies
		      \begin{equation}\label{eq:main-ll}
			      \int_{\mathbb{R}^D} D_{\mathrm{LL}}\!\bigl(\nabla u^*,\nabla u^y\bigr)\,\rho_\lambda^Y(dy)
			      \;\le\;
			      2\,C_\Phi C_\ell\,C_{\mathrm{fr}}\,\lambda .
		      \end{equation}

		\item (Distance to the MFG equilibrium.) Let $\mathcal{C},\mathcal{R}$ be the constants in the reflection-coupling estimate for drifts bounded by $C_b$ and Lipschitz with constant $C_L$.
		      Set
		      \[
			      \Gamma
			      :=
			      \frac{\kappa_2^c\kappa_3^c}{\mathcal{R}}(1+C_b)^{\theta_c}\,\mathcal{C}.
		      \]
		      Then
		      \begin{equation}\label{eq:main-w1}
			      \mathcal{W}_1\!\bigl(\rho_\lambda,\ m^*\otimes\rho_\lambda^{Y}\bigr)
			      \;\le\;
			      C_{\mathrm{fr}}\,\lambda
			      \;+\;
			      \Gamma\,\sqrt{8\kappa_2^c\,C_\Phi C_\ell\,C_{\mathrm{fr}}}\ \lambda^{1/2}.
		      \end{equation}
		      In particular, the asymptotic error is of order $O(\lambda^{1/2})$ as $\lambda\to0$.
	\end{enumerate}
\end{thm}
The proof of this main theorem will be  reported  in Section~\ref{proof-thm-2}.

\begin{rem}[Lasry–Lions divergence]
    The Lasry--Lions divergence is used to compare the frozen equilibria $\widehat{\rho}_y$ with the Nash equilibrium $m^*$. 
    Under the coercivity of $H$, it yields a control of the feedback error, which is then converted into a Wasserstein estimate. As demonstrated by the numerical examples in the next section, the divergence also serves as a natural metric for quantifying the error between the SFP invariant measure and the MFG equilibrium.
\end{rem}

\begin{rem}[Sharpness of the $\lambda^{1/2}$ rate]\label{rem:sharpness}
	The bound \eqref{eq:main-w1} yields an $O(\lambda^{1/2})$ convergence rate as $\lambda\to0$.
	This rate is \emph{sharp} in general: in the linear--quadratic example of
	Subsection~\ref{subsection-ex1}, the stationary law of the SFP dynamics satisfies
	\[
		\mathcal{W}_2\!\bigl(\rho_\lambda,\ m^*\otimes\rho_\lambda^Y\bigr)\asymp \lambda^{1/2}
		\qquad\text{as }\lambda\to0,
	\]
	and hence the same order holds for $\mathcal{W}_1$ (since $\mathcal{W}_1\le \mathcal{W}_2$).
\end{rem}

\section{Examples}\label{section-examples}

\subsection{Linear--quadratic case}\label{subsection-ex1}

For clarity we present a one–dimensional example on $\mathbb{R}$ (rather than on the torus), for which both the ergodic MFG and the SFP dynamics can be solved explicitly. The example also illustrates that the $O(\sqrt{\lambda})$ accuracy is sharp.

\medskip
\noindent\textbf{Model.}
We consider the quadratic running cost and a cylindrical interaction of the form
\[
	c(x,\alpha)=\frac12\,\alpha^2,
	\qquad
	\ell(r)=(r,\tfrac12 r^2),
	\qquad
	\Phi(y_1,y_2)=\frac12 y_1^2+y_2,
	\qquad x,r,\alpha\in\mathbb{R}.
\]
Then $\nabla\Phi(y)=(y_1,1)$ and therefore, for $m\in\mathcal{P}(\mathbb{R})$,
\[
	f(m,x)
	=
	\nabla\Phi(\langle \ell,m\rangle)\cdot \ell(x)
	=
	\Bigl(\int r\,m(dr)\Bigr)\,x+\frac12 x^2.
\]
Denote $\mu:=\int r\,m(dr)$.
For a frozen population $m$, the ergodic control problem reads
\[
	\delta^m
	=
	\lim_{T\to\infty}\frac1T
	\inf_{\alpha}
	\mathbb{E}\Bigl[\int_0^T\Bigl(\frac12\alpha_t^2+\mu X_t+\frac12 X_t^2\Bigr)\,dt\Bigr],
	\qquad
	dX_t=\alpha_t\,dt+\sqrt2\,dW_t.
\]
The Hamiltonian is
\[
	H(x,p)=\inf_{\alpha\in\mathbb{R}}\bigl(\alpha p+\tfrac12\alpha^2\bigr)
	=-\frac12 p^2,
	\qquad
	H_p(x,p)=-p,
\]
so the optimal feedback is $\alpha^*(x)=-u'(x)$.

\medskip
\noindent\textbf{Solving the ergodic HJB--FP system.}
The stationary MFG system is
\begin{equation}\label{eq-ex-HJB-FP}
	\begin{cases}
		-\delta + u''(x)-\frac12\bigl(u'(x)\bigr)^2 + \mu x+\frac12 x^2=0, \\[1mm]
		\partial_x\!\bigl(u'(x)\,m(x)\bigr)+m''(x)=0,
	\end{cases}
\end{equation}
where $\mu=\int x\,m(dx)$.

Fix $m$ (hence $\mu$) and seek a quadratic corrector
\[
	u^m(x)=\frac a2 x^2+bx + \text{const}.
\]
Then $u'(x)=ax+b$ and $u''(x)=a$. Plugging into the HJB equation gives
\[
	-\delta + a -\frac12(ax+b)^2 + \mu x+\frac12 x^2=0,
\]
hence identification of coefficients yields
\[
	-\frac12 a^2+\frac12=0 \ \Rightarrow\ a=\pm1,
	\qquad
	-ab+\mu=0 \ \Rightarrow\ b=\frac{\mu}{a}.
\]
To obtain an ergodic (mean-reverting) controlled dynamics on $\mathbb{R}$ we must select $a=1$, so $b=\mu$, and
\begin{equation}\label{eq:u-m-explicit}
	u^m(x)=\frac12 x^2+\mu x+\text{const},
	\qquad
	\delta^m=1-\frac12\mu^2.
\end{equation}

The associated optimal drift is
\[
	\alpha^m(x)=-u^{m\,\prime}(x)=-(x+\mu),
\]
so the optimally controlled process is an Ornstein--Uhlenbeck diffusion
\[
	dX_t=-(X_t+\mu)\,dt+\sqrt2\,dW_t,
\]
whose invariant law is $\mathcal N(-\mu,1)$. Therefore the mean of the invariant law equals $-\mu$, and the MFG fixed-point condition $\mu=\int x\,m(dx)$ enforces $\mu=-\mu$, i.e.\ $\mu^*=0$. Consequently,
\[
	m^*=\mathcal N(0,1),
	\qquad
	u^*(x)=\frac12 x^2+\text{const}.
\]

\medskip
\noindent\textbf{SFP dynamics.}
For the SFP update $dm_t=\lambda(\delta_{X_t}-m_t)\,dt$, the moment
\[
	Y_t:=\int x\,m_t(dx)
\]
satisfies $dY_t=\lambda(X_t-Y_t)\,dt$.
Since the optimal drift depends on $m_t$ only through $\mu_t=\int x\,m_t(dx)=Y_t$, the reduced SFP system is
\begin{equation}\label{eq-ex-SFP}
	\begin{cases}
		dX_t=-(X_t+Y_t)\,dt+\sqrt2\,dW_t, \\
		dY_t=\lambda(X_t-Y_t)\,dt.
	\end{cases}
\end{equation}
This is a linear Gaussian diffusion on $\mathbb{R}^2$ with drift matrix
\[
	A=\begin{pmatrix}-1 & -1\\ \lambda & -\lambda\end{pmatrix},
	\qquad
	d\binom{X_t}{Y_t}=A\binom{X_t}{Y_t}\,dt+\binom{\sqrt2}{0}\,dW_t.
\]

\medskip
\noindent\textbf{Invariant law and $O(\sqrt\lambda)$ sharpness.}
The process $(X_t,Y_t)$ is ergodic for every $\lambda>0$ (both eigenvalues of $A$ have negative real part). Its unique invariant law $\rho_\lambda$ is centered Gaussian. Writing its covariance as
\[
	\Sigma_\lambda=
	\begin{pmatrix}
		\Var_{\rho_\lambda}(X)   & \Cov_{\rho_\lambda}(X,Y) \\
		\Cov_{\rho_\lambda}(X,Y) & \Var_{\rho_\lambda}(Y)
	\end{pmatrix},
\]
the Lyapunov equation $A\Sigma_\lambda+\Sigma_\lambda A^\top + BB^\top=0$ with
$B=(\sqrt2,0)^\top$ yields the explicit solution
\begin{equation}\label{eq:Sigma-explicit}
	\Var_{\rho_\lambda}(X)=\frac{\lambda+2}{2(\lambda+1)},
	\qquad
	\Cov_{\rho_\lambda}(X,Y)=\frac{\lambda}{2(\lambda+1)},
	\qquad
	\Var_{\rho_\lambda}(Y)=\frac{\lambda}{2(\lambda+1)}.
\end{equation}
In particular, as $\lambda\to0^+$,
\[
	\Var_{\rho_\lambda}(X)=1-\frac{\lambda}{2}+O(\lambda^2),
	\qquad
	\Var_{\rho_\lambda}(Y)=\frac{\lambda}{2}+O(\lambda^2),
	\qquad
	\Cov_{\rho_\lambda}(X,Y)=\frac{\lambda}{2}+O(\lambda^2).
\]

Now compare $\rho_\lambda$ to the product measure $m^*\otimes \rho_\lambda^Y$, where $m^*=\mathcal N(0,1)$ and $\rho_\lambda^Y$ is the $Y$-marginal of $\rho_\lambda$ (a centered Gaussian with variance $\Var_{\rho_\lambda}(Y)$). Both measures are Gaussian with zero mean; their covariance matrices are
\[
	\Sigma_1=\Sigma_\lambda,
	\qquad
	\Sigma_2=\begin{pmatrix}1 & 0\\ 0 & \Var_{\rho_\lambda}(Y)\end{pmatrix}.
\]
The squared $\mathcal{W}_2$ distance between centered Gaussians satisfies
\[
	\mathcal{W}_2^2(\mathcal N(0,\Sigma_1),\mathcal N(0,\Sigma_2))
	=
	\Tr\!\Bigl(\Sigma_1+\Sigma_2-2(\Sigma_2^{1/2}\Sigma_1\Sigma_2^{1/2})^{1/2}\Bigr).
\]
Using \eqref{eq:Sigma-explicit} and a Taylor expansion at small $\lambda$ (equivalently, using that the Frobenius norm of $\Sigma_1-\Sigma_2$ is $O(\lambda)$ while $\Sigma_2$ stays uniformly non-degenerate on the $X$-block), one obtains
\[
	\mathcal{W}_2(\rho_\lambda,\ m^*\otimes \rho_\lambda^Y)=O(\sqrt{\lambda})
	\qquad\text{as }\lambda\to0^+.
\]
Thus the limiting discrepancy between the stationary SFP law and the equilibrium product law is of order $\sqrt{\lambda}$.

\subsection{Numerical examples}\label{subsection-ex2}

\subsubsection{General experimental setup}
Our numerical experiments aim to illustrate the convergence of the SFP algorithm towards the ergodic MFG Nash equilibrium. The algorithm approximates the theoretical equilibrium distribution $m^*$ by simulating a population of $N$ (non-interacting) particles tracking the equilibrium dynamics, and a scalar belief $Y_t$ tracking the observation of each particle.

For each example, the SFP algorithm was simulated using the Euler--Maruyama method with timestep $\Delta t$ for the dynamics of $X_t$. The value function $u^{Y_t}$ is either obtained using an analytical formula (in the LQ model), or by solving the instantaneous ergodic HJB equation using a finite-difference scheme (for general models).

To evaluate convergence, the following classes of metrics are computed over time:
\begin{itemize}
	\item \textbf{Lasry--Lions divergence}: This quantity estimates the distance between the instantaneous feedback $- \nabla u^{Y_t}$ and the theoretical optimal feedback $-\nabla u^*$. This tracks the control policy error formalized in \eqref{eq:main-ll}. This quantity is denoted $D_{\text{LL}}$ in the figures. 
	\item \textbf{Wasserstein distance}: The joint error on the product space between the physical particles and their internal belief tracking is evaluated by aggregating the results across independent random seeds. Because the theoretical limit of the parameter distribution (as $\lambda \to 0$) is a Dirac mass, the $L^1$ optimal transport cost separates as $\mathcal{W}_1(\text{Hist}^X, m^*) + \mathbb{E}[\|Y_t - Y^*\|_1]$. Here, $\text{Hist}^X$ denotes the distribution of the spatial particle $X_t$. We approximate this histogram and the expectation for the distance in $Y$ using the empirical simulation of particles. This quantity is referred to as ``$W_1$ Joint of Mean Hist'' in the figures.
\end{itemize}
To obtain robust statistics and mitigate noise, the reported metrics are averaged across 15 independent simulation runs (with different seens in the implementation) initialized with different random seeds. Specifically, we present two types of results:
\begin{itemize}
	\item \textbf{Scaling of long-time $D_{\text{LL}}$ and Wasserstein distance with respect to $\lambda$:} For the scaling analysis (Figures~\ref{fig:lq_scaling} and~\ref{fig:cosine_scaling}), the reported time-asymptotic $D_{\text{LL}}$ for each $\lambda$ is obtained by computing the cross-seed mean and standard deviation at each timestep. To smoothen the results, we subsequently average these statistics over the terminal quasi-stationary window (the final 10\% of the simulation). Error bars denote the cross-seed variability. The reported $W_1$ value for each $\lambda$ is computed by first averaging the empirical distributions across all 15 seeds to obtain a mean distribution, and then calculating the Wasserstein distance between this aggregated distribution and the theoretical target equilibrium $m^*$.
	\item \textbf{Time evolution of $Y_t$, $D_{\text{LL}}$ and Wasserstein distance:} The time-series plots in Figures~\ref{fig:lq_evol} and~\ref{fig:cosine_evol} display the mean across 15 independent runs of $N$ particles, with shaded areas representing one standard deviation. For the $W_1$ distance, we report both the distance on the joint distribution (as described above, in dark red line) and the distance over the first marginal ($X$) only (orange dashed line).
\end{itemize}

\subsubsection{Linear--quadratic case}
We first consider the LQ case introduced in Section~\ref{subsection-ex1}. The initial state for all particles is drawn from a Gaussian $\mathcal{N}(X_0, \Sigma_0^2)$, and the scalar belief is initialized at $Y_0$.
The parameter values and the update rates $\lambda$ used in the simulation are listed in Table~\ref{tab:lq_params}. Note that simulation steps were adaptively increased for smaller values of $\lambda$ to guarantee that convergence in the slow phase was adequately reached (from $30000$ steps for $\lambda = 1.0$ to $600000$ steps for $\lambda = 0.01$).
\begin{table}[htbp]
	\centering
	\begin{tabular}{|l|c|}
		\hline
		\textbf{Parameter}                  & \textbf{Value}              \\ \hline
		Population size $N$                 & $500$                       \\
		Simulation steps                    & $30000$ to $600000$       \\
		Spatial grid size $N_x$             & $500$                       \\
		Timestep $\Delta t$                 & $0.001$                     \\
		Initial state $X_0$                 & $1.0$                       \\
		Initial state variance $\Sigma_0^2$ & $0.01^2$                    \\
		Initial belief $Y_0$                & $0.5$                       \\
		Update rates $\lambda$              & $1.0, 0.3, 0.1, 0.03, 0.01$ \\ \hline
	\end{tabular}
	\caption{Parameters for the SFP simulation in the LQ case.}
	\label{tab:lq_params}
\end{table}

Figure~\ref{fig:lq_scaling} illustrates the scaling of the time-asymptotic errors (Lasry--Lions and Wasserstein metrics) with respect to the learning rate $\lambda$. Fitting a line in log-log scale provides an estimate of the convergence rate: we observe a rate of order $\lambda^{0.44}$, which is close to the theoretical rate of $\sqrt{\lambda}$ derived in Subsection~\ref{subsection-ex1}. The small difference can probably be explained by numerical errors due to time discretization and Monte Carlo samples.

To highlight the dynamics leading to these asymptotic values, we display the trace of the metrics over time for a large learning rate ($\lambda=1.0$) and a small learning rate ($\lambda=0.01$) in Figure~\ref{fig:lq_evol}. The convergence is slower but sharper for smaller $\lambda$. In both cases, we observe in the right column of Figure~\ref{fig:lq_evol} that the Wasserstein distance for the state component reaches lower values than the Wasserstein distance for the joint distribution of $(X,Y)$.

\begin{figure}[htbp]
	\centering
	\includegraphics[width=0.48\textwidth]{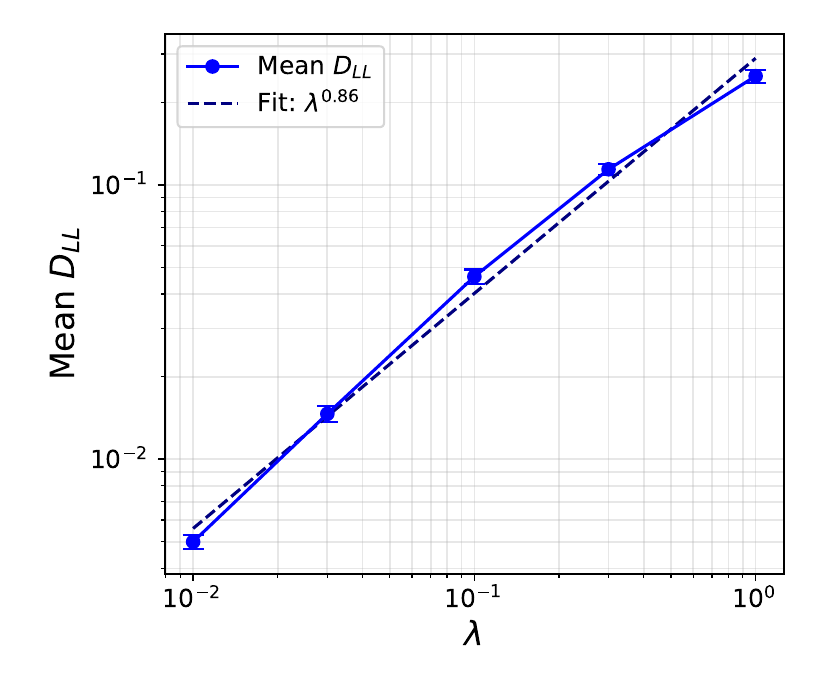} \hfill
	\includegraphics[width=0.48\textwidth]{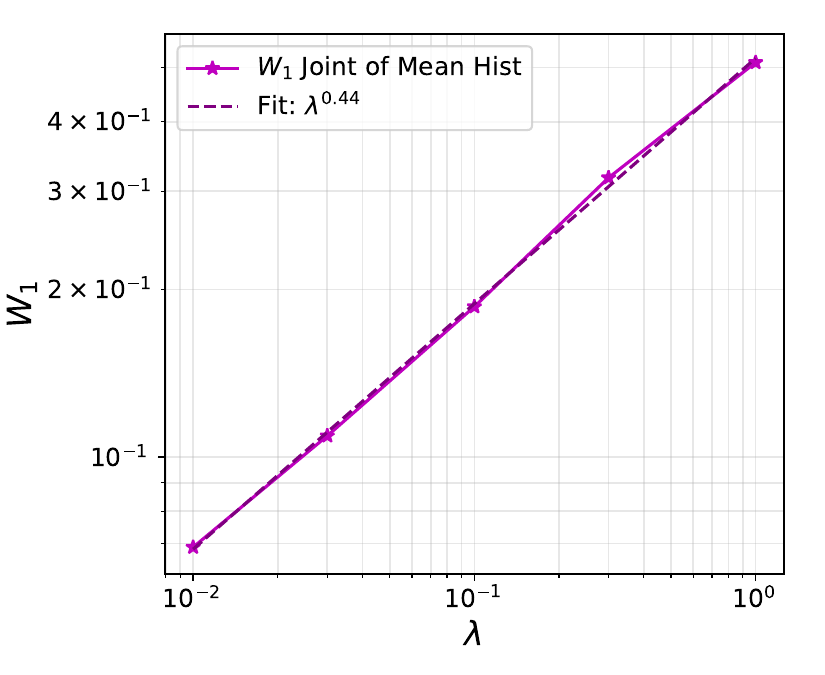}
	\caption{Time-asymptotic Lasry--Lions divergence $D_{\text{LL}}$ and Wasserstein distance $W_1$ as a function of the learning rate $\lambda$ in the linear--quadratic case.}
	\label{fig:lq_scaling}
\end{figure}

\begin{figure}[htbp]
	\centering
	\begin{tabular}{c@{\hspace{2pt}}cccc}
		(a)                                                                                                                                                              & \includegraphics[width=0.22\textwidth]{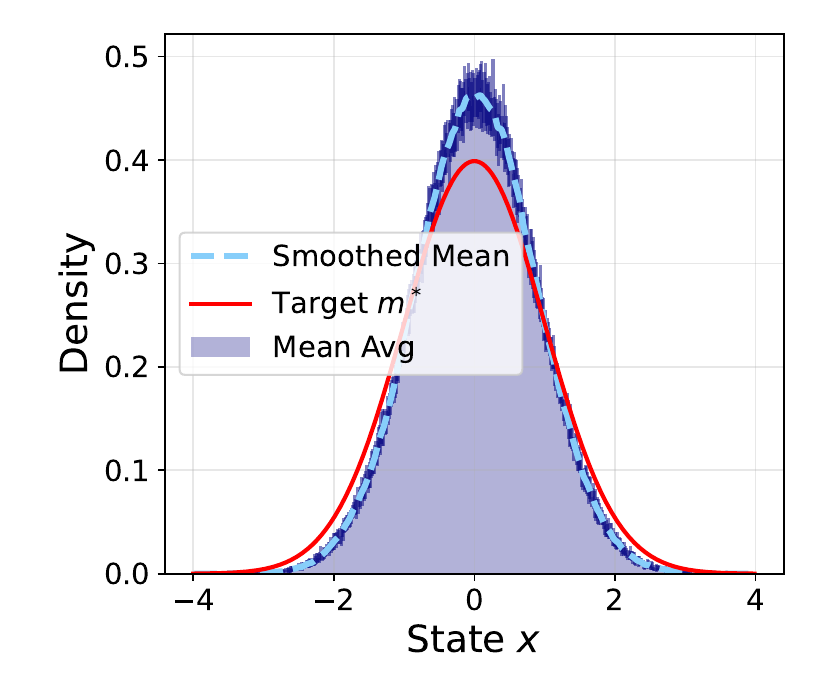}  &
		\includegraphics[width=0.22\textwidth]{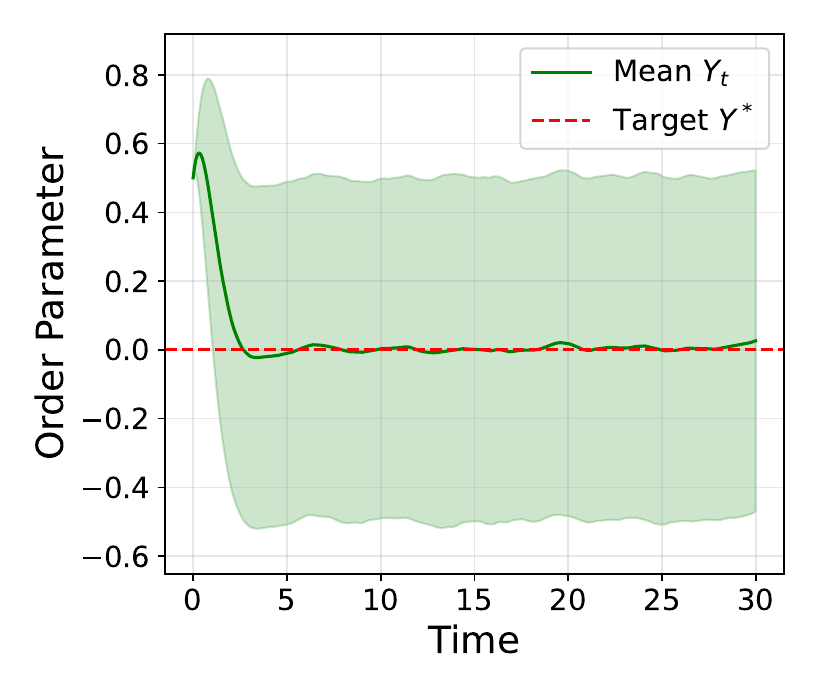}    &
		\includegraphics[width=0.22\textwidth]{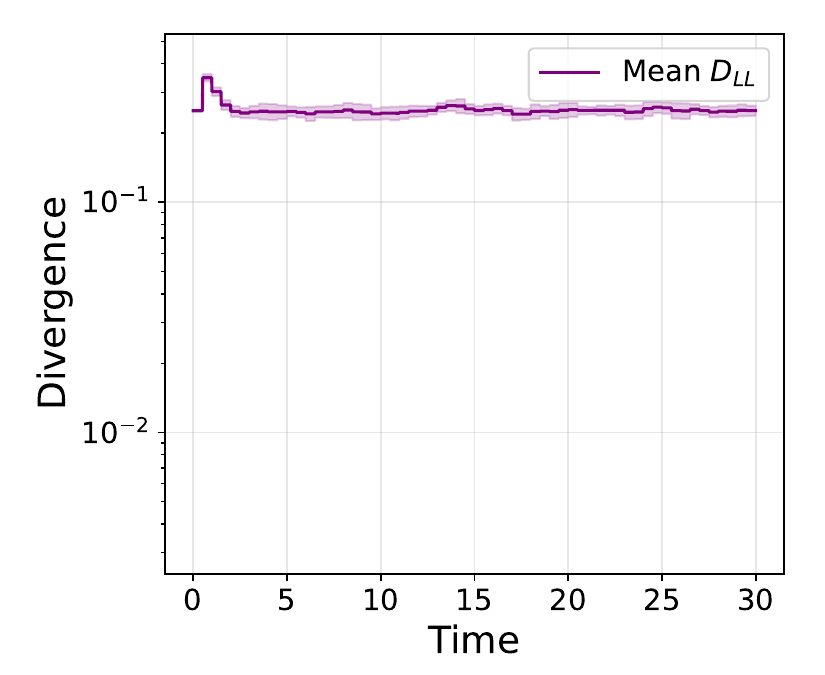}  &
		\includegraphics[width=0.22\textwidth]{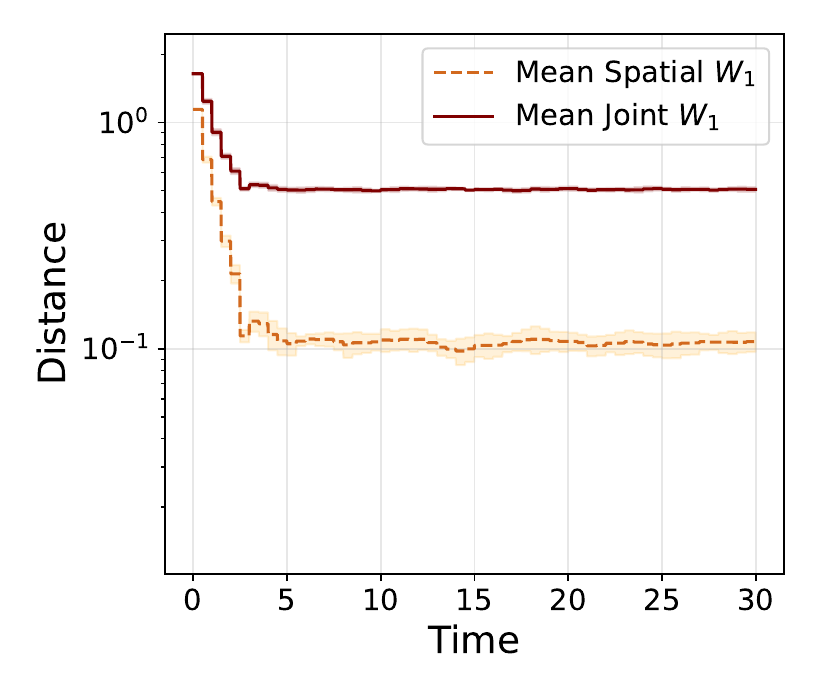}                                                                                                                                                                         \\[1ex]
		(b)                                                                                                                                                              & \includegraphics[width=0.22\textwidth]{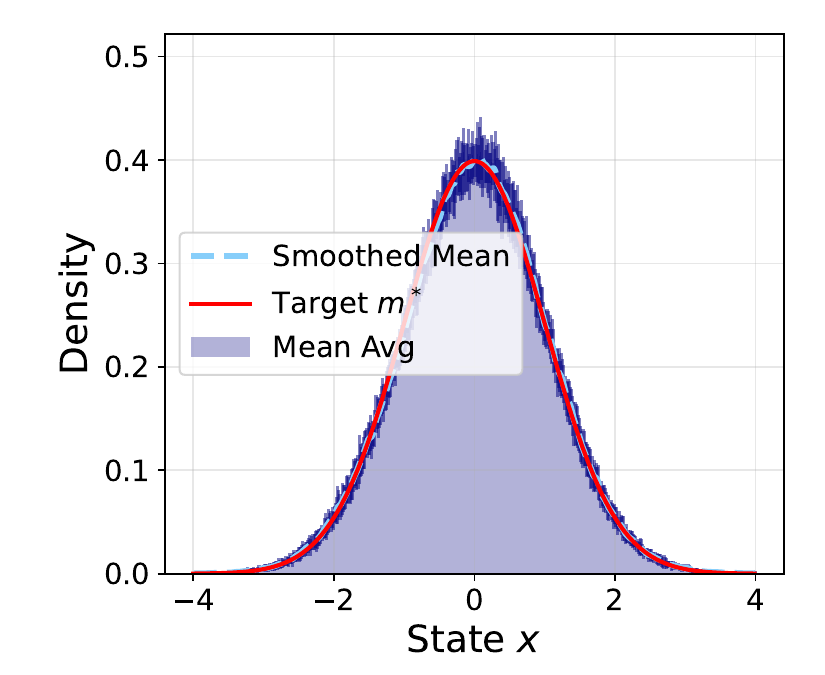} &
		\includegraphics[width=0.22\textwidth]{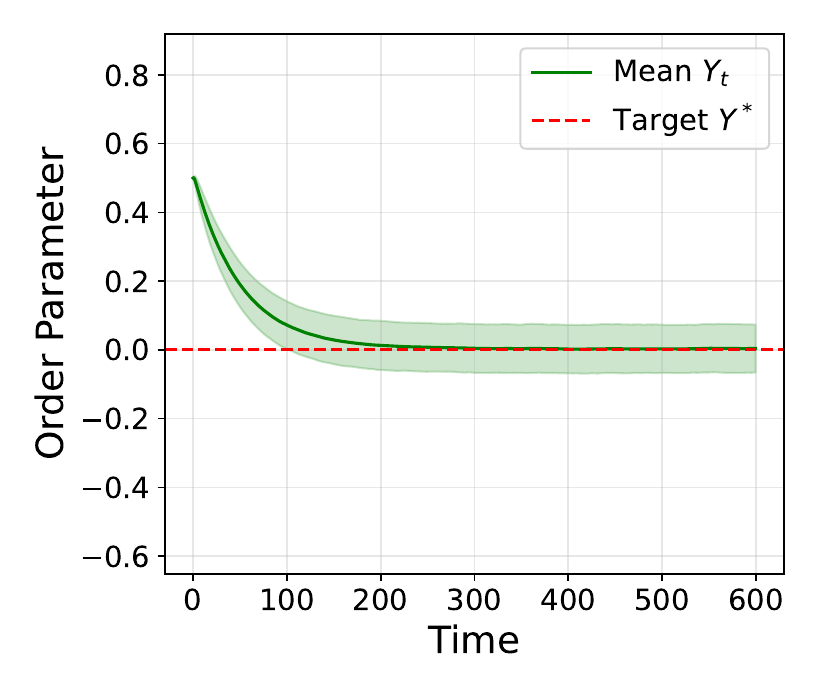}   &
		\includegraphics[width=0.22\textwidth]{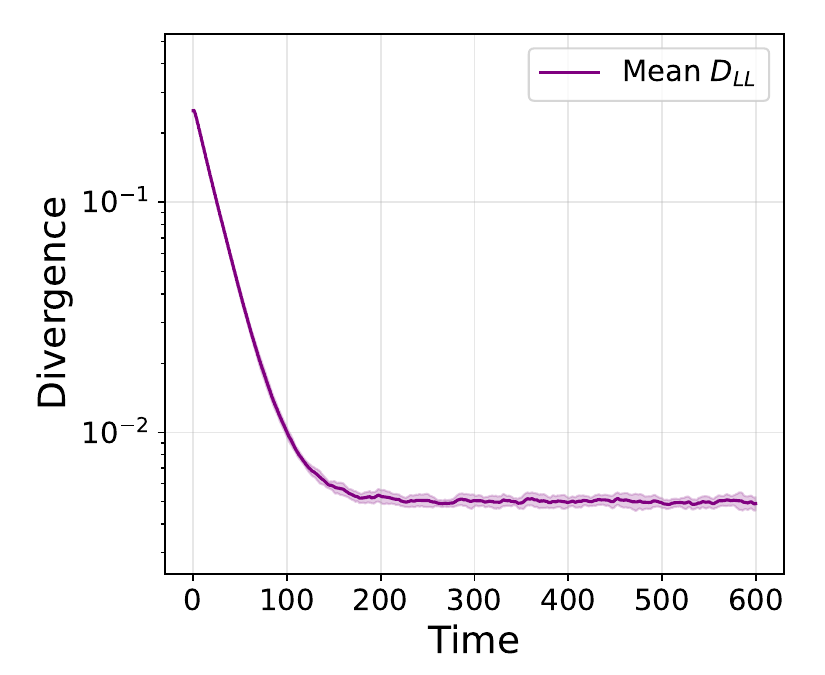} &
		\includegraphics[width=0.22\textwidth]{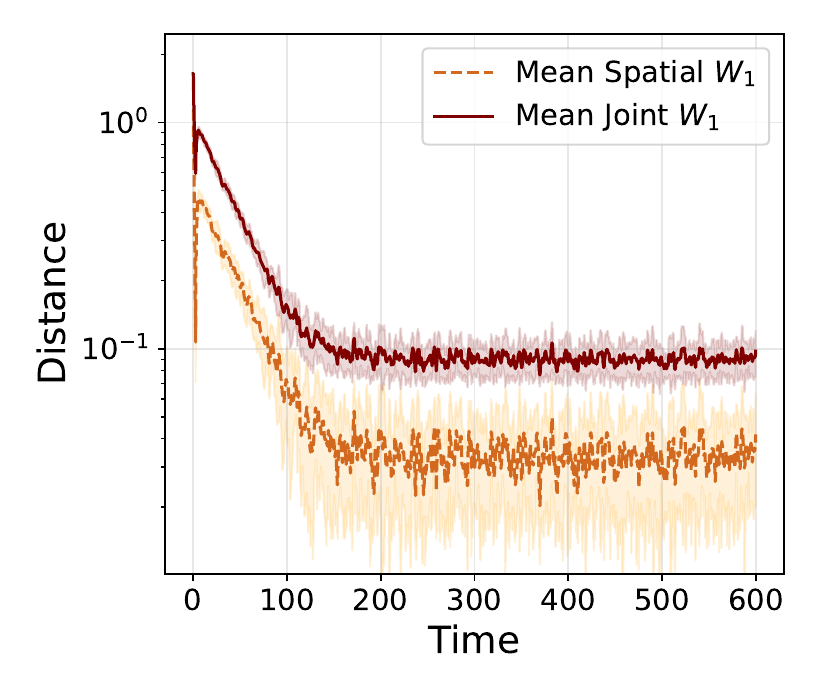}
	\end{tabular}
	\caption{Evolution of the empirical distribution, belief parameter $Y_t$, Lasry--Lions divergence, and Wasserstein distance in the LQ case for: (a) large rate $\lambda=1.0$, and (b) small rate $\lambda=0.01$. Shaded areas indicate one standard deviation across simulation seeds.}
	\label{fig:lq_evol}
\end{figure}

\subsubsection{Cosine model}
Our second numerical test investigates a manufactured example with periodic boundary conditions on $\mathbb{T} = [0,1]$. We engineer the problem such that the explicit analytical solution is $u^*(x) = \cos(2\pi x)$ and the invariant distribution is $m^*(x) \propto \exp(-\cos(2\pi x)/\epsilon)$. Specifically, the interaction cost is chosen as a cosine function augmented by a parameter $\mu$:
\[
	c(x, \alpha) = \frac12 \alpha^2 -\epsilon u^{*\prime\prime}(x) +  \frac{1}{2}\bigl(u^{*\prime}(x)\bigr)^2 - f(m^*, x),
	\qquad f(m, x) = \mu Y \cos(2\pi x),
\]
where $Y = \int \cos(2\pi x) m(dx)$ and $\epsilon$ sets the scale of intrinsic noise.

The parameters for the simulations are indicated in Table~\ref{tab:cosine_params}. 

\begin{table}[htbp]
	\centering
	\begin{tabular}{|l|c|}
		\hline
		\textbf{Parameter}         & \textbf{Value}              \\ \hline
		Population size $N$        & $500$                       \\
		Simulation steps           & $30000$ to $800000$       \\
		Spatial grid size $N_x$    & $500$                       \\
		Timestep $\Delta t$        & $0.001$                     \\
		Noise scale $\epsilon$     & $1.0$                       \\
		Interaction strength $\mu$ & $1.0$                       \\
		Initial belief $Y_0$       & $0.5$                       \\
		Update rates $\lambda$     & $1.0, 0.3, 0.1, 0.03, 0.01$ \\ \hline
	\end{tabular}
	\caption{Parameters for the SFP simulation in the Cosine model.}
	\label{tab:cosine_params}
\end{table}

The resulting errors as a function of the learning rate $\lambda$ are plotted in Figure~\ref{fig:cosine_scaling}. Figure~\ref{fig:cosine_evol} showcases the evolution of the relevant metrics with a high rate ($\lambda=1.0$) versus a lower rate ($\lambda=0.01$). In this example, we observe that even with large $\lambda$, the equilibrium distribution on $X$ is learned very accurately. This is reflected in the fact that the Wasserstein distance is very small even for large $\lambda$. However, joint Wasserstein distance and the Lasry--Lions divergence take large values when $\lambda$ is large, and achieve lower values as $\lambda$ decreases.

\begin{figure}[htbp]
	\centering
	\includegraphics[width=0.48\textwidth]{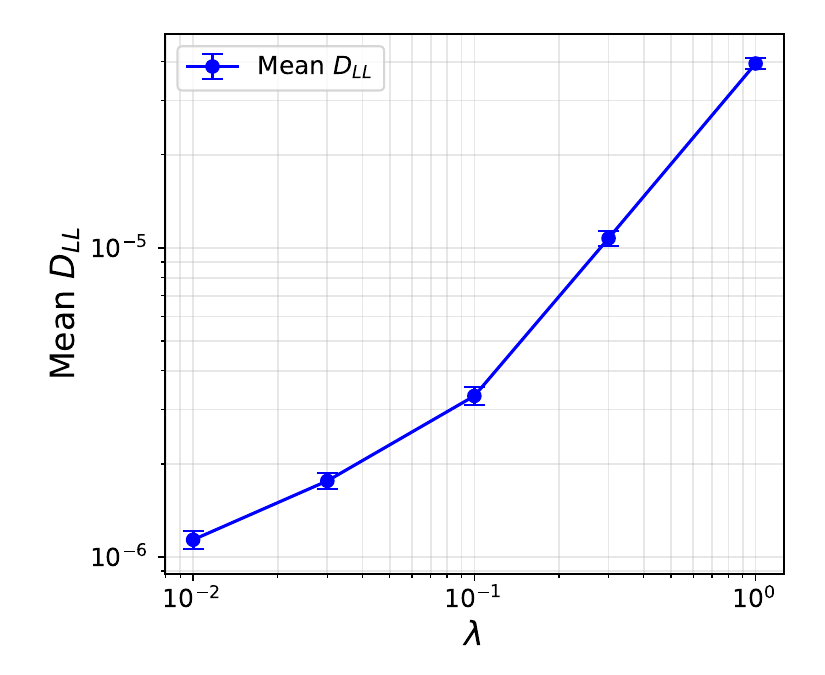} \hfill
	\includegraphics[width=0.48\textwidth]{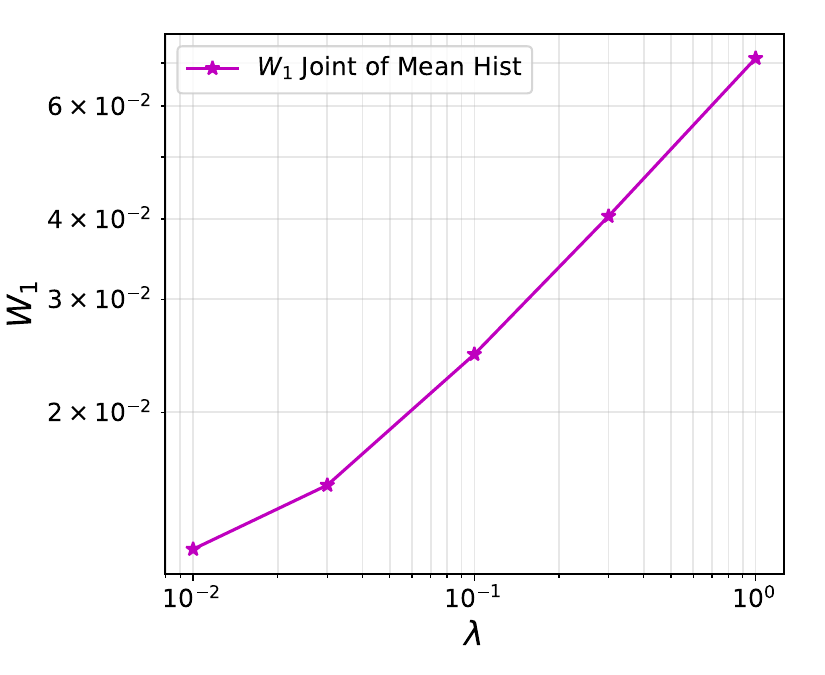}
	\caption{Time-asymptotic Lasry--Lions divergence $D_{\text{LL}}$ and Wasserstein distance $W_1$ as a function of $\lambda$ in the Cosine model.}
	\label{fig:cosine_scaling}
\end{figure}

\begin{figure}[htbp]
	\centering
	\begin{tabular}{c@{\hspace{2pt}}cccc}
		(a)                                                                                                                                                                   & \includegraphics[width=0.22\textwidth]{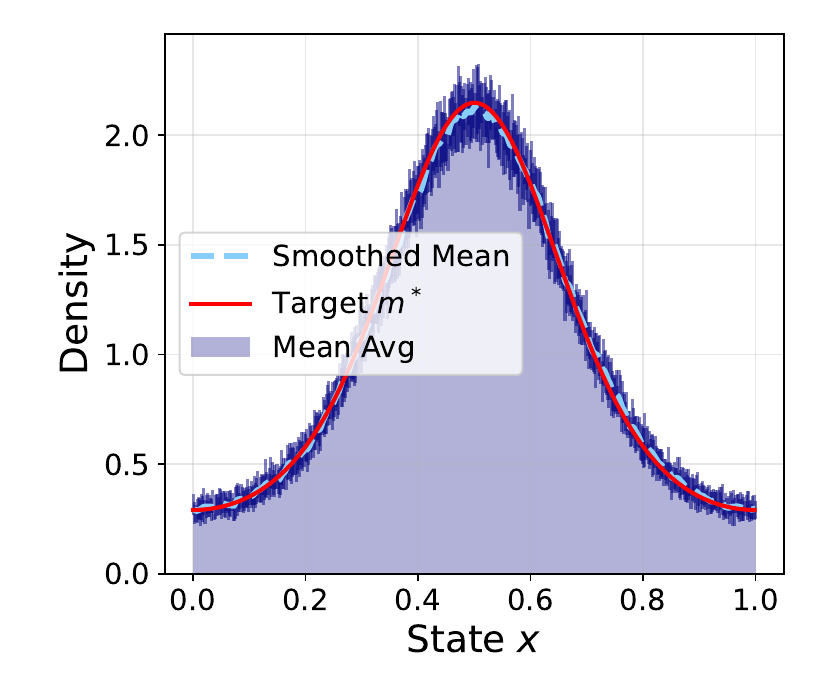}  &
		\includegraphics[width=0.22\textwidth]{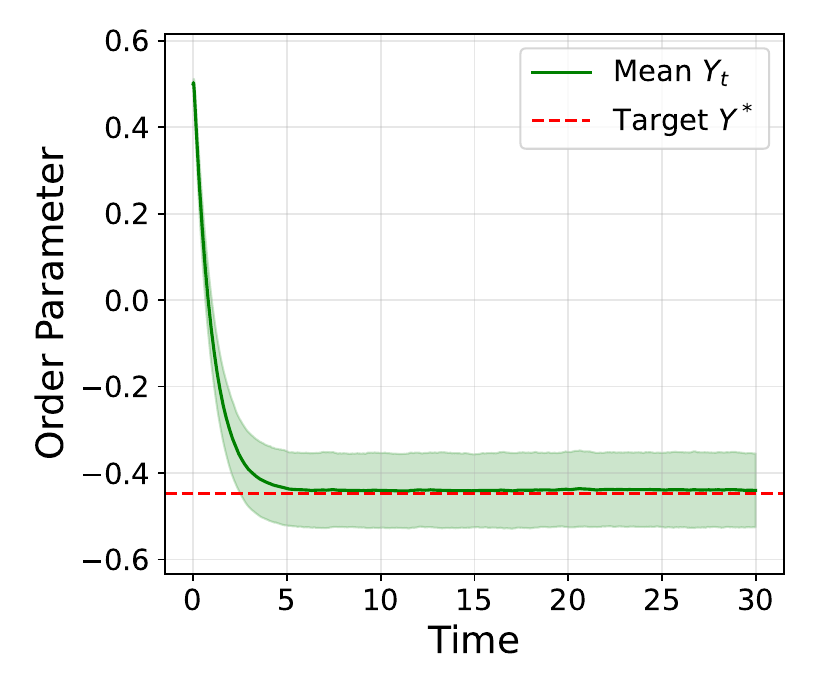}    &
		\includegraphics[width=0.22\textwidth]{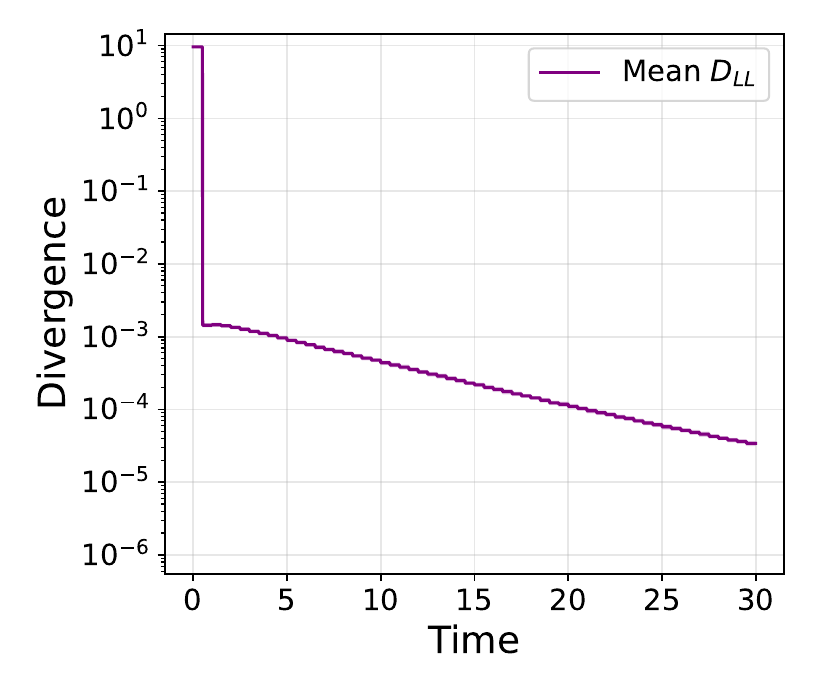}  &
		\includegraphics[width=0.22\textwidth]{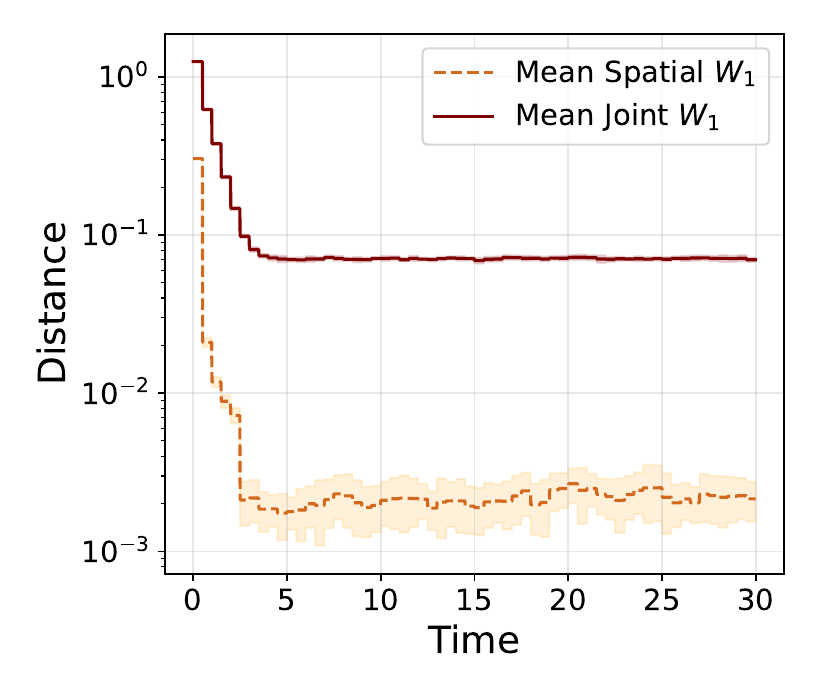}                                                                                                                                                                              \\[1ex]
		(b)                                                                                                                                                                   & \includegraphics[width=0.22\textwidth]{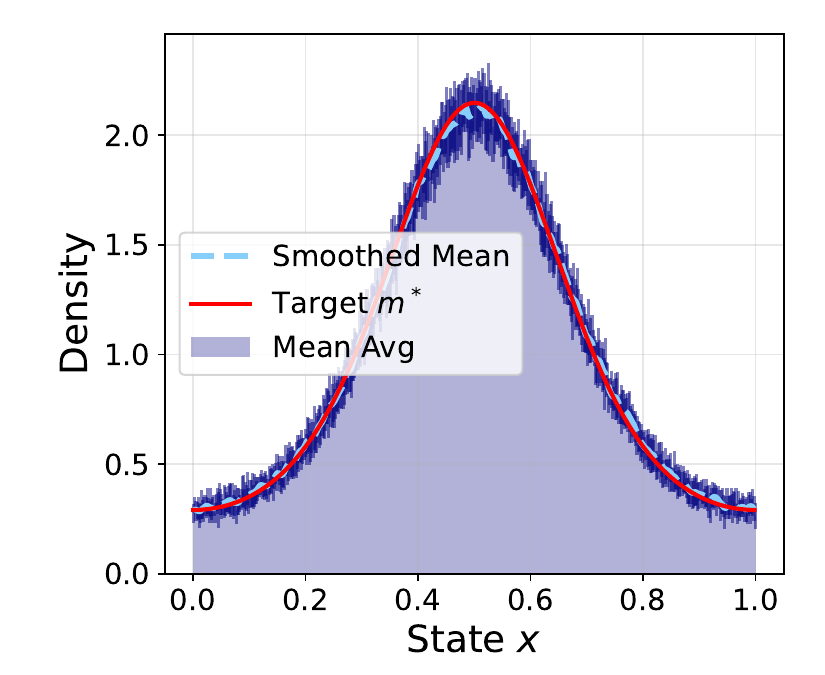} &
		\includegraphics[width=0.22\textwidth]{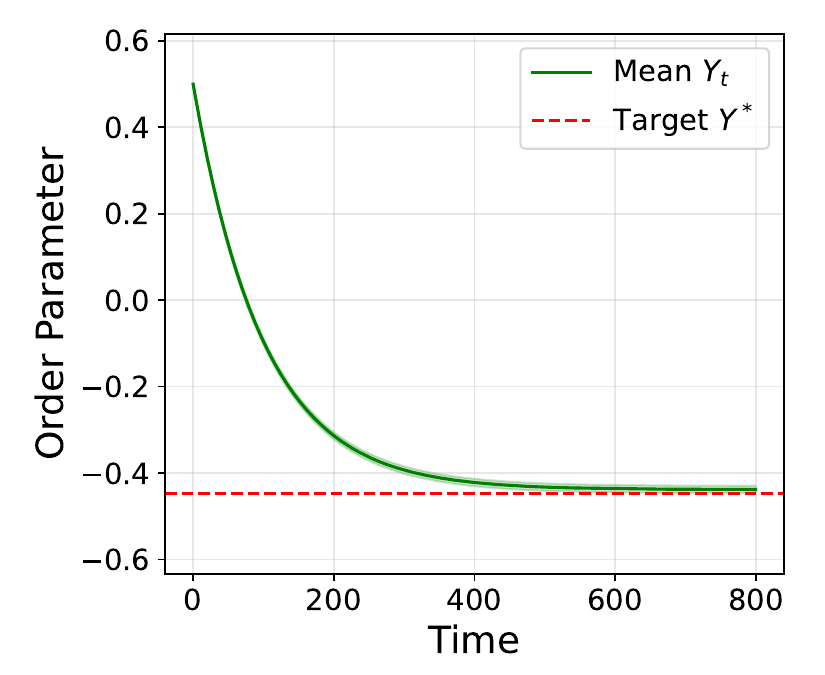}   &
		\includegraphics[width=0.22\textwidth]{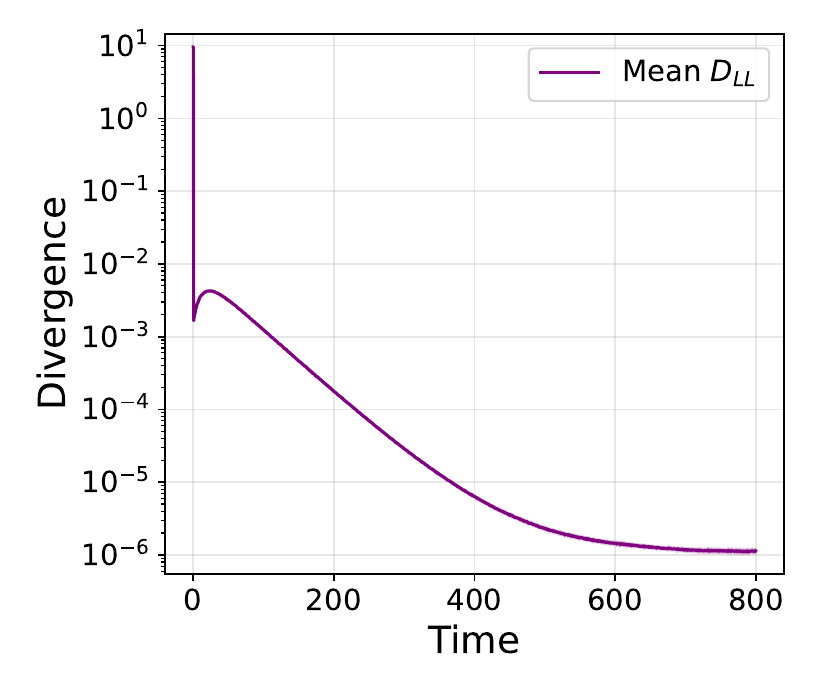} &
		\includegraphics[width=0.22\textwidth]{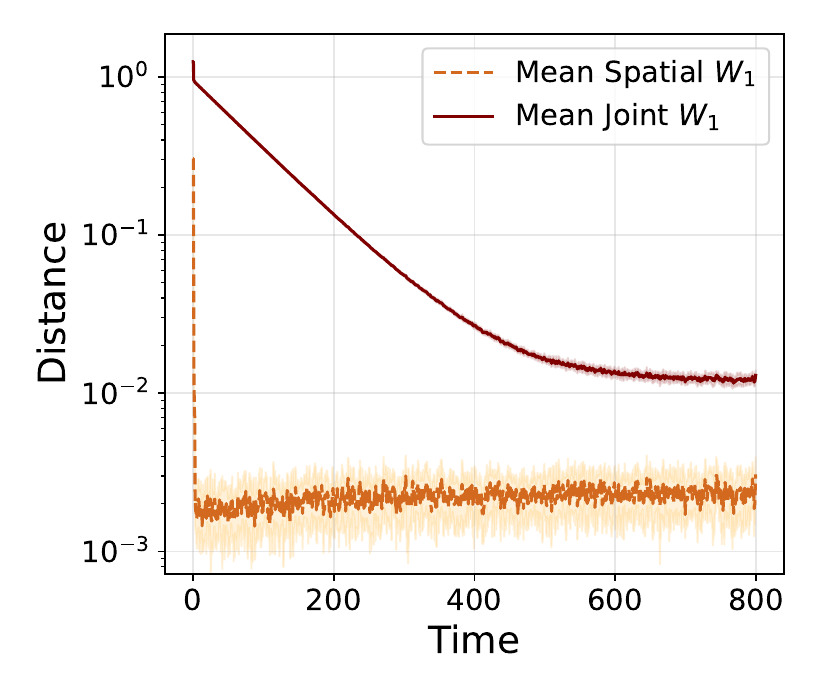}
	\end{tabular}
	\caption{Evolution of the model quantities in the Cosine model for: (a) large rate $\lambda=1.0$, and (b) small rate $\lambda=0.01$. Shaded areas indicate one standard deviation across simulation seeds.}
	\label{fig:cosine_evol}
\end{figure}

\section{Proofs}\label{section-proof}

\subsection{Regularity estimates}\label{subsec:proof:reg}

This section establishes the regularity of the drift field
\[
	(x,y)\longmapsto b^y(x):=H_p\bigl(x,\nabla u^y(x)\bigr),
\]
which is the key input for the coupling arguments used later in the analysis of the SFP dynamics~\eqref{eq-SFP-XY}. The goal is to prove Lemma~\ref{lem-regularity}: uniform boundedness and joint Lipschitz continuity of $b^y(x)$ in $(x,y)$.

\medskip
\paragraph{Finite-horizon auxiliary problem.}
Fix $T>0$ and $y\in\mathbb{R}^D$. Consider the finite-horizon stochastic control problem
\begin{equation}\label{eq-T-mfg}
	u^{T,y}(t,x)
	=
	\inf_{\alpha}\;
	\mathbb{E}\Bigl[
		\int_t^T \Bigl(c(X_s,\alpha_s)+V(X_s,y)\Bigr)\,ds
		\;\Big|\; X_t=x
		\Bigr],
	\qquad
	dX_s=\alpha_s\,ds+\sqrt2\,dW_s,
\end{equation}
with terminal cost $u^{T,y}(T,\cdot)\equiv0$. (We emphasize that $V(x,y)=\nabla\Phi(y)\cdot \ell(x)$ and is uniformly bounded and Lipschitz in $x$, uniformly over bounded $y$.) The value function $u^{T,y}$ solves the HJB equation
\begin{equation}\label{eq-T-hjb}
	\partial_t u^{T,y}(t,x)+\Delta u^{T,y}(t,x)+H\bigl(x,\nabla_x u^{T,y}(t,x)\bigr)+V(x,y)=0,
	\qquad
	u^{T,y}(T,x)=0.
\end{equation}

\begin{lem}[Uniform bound on $\nabla_x u^{T,y}$]\label{lem-gradiantu-bound}
	Assume Assumptions~\ref{ass-potential} and~\ref{ass-cost}. There exists a constant $C_{\nabla u}>0$ such that for every horizon $T>0$, every $y\in\mathbb{R}^D$, and every $t\in[0,T]$,
	\[
		\|\nabla_x u^{T,y}(t,\cdot)\|_\infty \le C_{\nabla u}.
	\]
	The constant $C_{\nabla u}$ depends only on $(\kappa_0^c,\kappa_1^c,\kappa_2^c,\kappa_3^c,\theta_c,C_\ell,C_\Phi)$.
\end{lem}

Lemma~\ref{lem-gradiantu-bound} is a direct consequence of Lemma~1.5 in~\cite{CardaliaguetPorretta2019ergodicMaster}. For completeness we provide a proof in Appendix~\ref{app-lem}.

\begin{cor}[Uniform boundedness of $H_p(x,\nabla u^{T,y})$]\label{cor-Hp-bounded}
	Assume Assumptions~\ref{ass-potential} and~\ref{ass-cost}. Then
	\[
		\sup_{T>0}\;\sup_{t\in[0,T]}\;\sup_{x\in\mathbb{T}^d}\;\sup_{y\in\mathbb{R}^D}
		\bigl|H_p\bigl(x,\nabla_x u^{T,y}(t,x)\bigr)\bigr|
		\le C_b,
	\]
	for a constant $C_b$ depending only on $(\kappa_1^c,\kappa_2^c,C_{\nabla u})$ (hence ultimately on the parameters in Lemma~\ref{lem-gradiantu-bound}).
\end{cor}

\begin{proof}
	For $(x,p)\in\mathbb{T}^d\times\mathbb{R}^d$, define the unique minimizer
	\[
		\alpha^*(x,p):=\arg\min_{\alpha\in\mathbb{R}^d}\{\alpha\cdot p+c(x,\alpha)\}.
	\]
	By strong convexity of $\alpha\mapsto c(x,\alpha)$ (Assumption~\ref{ass-cost}\textup{(ii)}), $\alpha^*(x,p)$ is well-defined and
	the envelope theorem gives
	\begin{equation}\label{eq-Hp-env}
		H_p(x,p)=\alpha^*(x,p).
	\end{equation}
	The first-order optimality condition reads
	\[
		p+c_\alpha\bigl(x,\alpha^*(x,p)\bigr)=0.
	\]
	Using strong convexity, for any $\alpha$ we have the standard estimate
	\[
		|c_\alpha(x,\alpha)|
		\ge \frac{1}{\kappa_2^c}|\alpha|-|c_\alpha(x,0)|.
	\]
	Taking $\alpha=\alpha^*(x,p)$ and using $c_\alpha(x,\alpha^*(x,p))=-p$ yields
	\[
		\frac{1}{\kappa_2^c}|\alpha^*(x,p)|-|c_\alpha(x,0)|
		\le |p|,
		\qquad\text{hence}\qquad
		|\alpha^*(x,p)|
		\le \kappa_2^c\bigl(|p|+\kappa_1^c\bigr).
	\]
	Combining \eqref{eq-Hp-env} with Lemma~\ref{lem-gradiantu-bound} gives
	\[
		\bigl|H_p\bigl(x,\nabla_x u^{T,y}(t,x)\bigr)\bigr|
		\le \kappa_2^c\bigl(C_{\nabla u}+\kappa_{1}^c\bigr)
		=:C_b,
	\]
	uniformly in $(T,t,x,y)$.
\end{proof}

\paragraph{A reflection coupling estimate.}
We next introduce the reflection coupling used to control differences of trajectories starting from different initial conditions. This estimate will later enter the Lipschitz bounds for $\nabla_x u^{T,y}(0,\cdot)$.

Fix $T>0$ and $y\in\mathbb{R}^D$ and define the time-dependent drift
\[
	b_t(x):=H_p\bigl(x,\nabla_x u^{T,y}(t,x)\bigr).
\]
Let $(X_t,\widehat X_t)$ be the reflection coupling defined up to the coupling time
\[
	T_0:=\inf\{t\ge0:\ X_t=\widehat X_t\}.
\]
For $t<T_0$,
\begin{equation}\label{eq-coupling1}
	\begin{aligned}
		dX_t          & = b_t(X_t)\,dt+\sqrt2\,dW_t,                   \\
		d\widehat X_t & = b_t(\widehat X_t)\,dt+\sqrt2\,d\widehat W_t,
	\end{aligned}
	\qquad
	d\widehat W_t=(I_d-2e_t e_t^\top)\,dW_t,
	\qquad
	e_t=\frac{X_t-\widehat X_t}{|X_t-\widehat X_t|}.
\end{equation}
For $t\ge T_0$ we set $\widehat X_t=X_t$.

\begin{lem}[Coupling probability bound]\label{lem-ref-coupling1}
	Assume Assumptions~\ref{ass-potential} and~\ref{ass-cost}. Let $(X_t,\widehat X_t)$ be the reflection coupling~\eqref{eq-coupling1} with coupling time
	\[
		T_0:=\inf\{t\ge0:\ X_t=\widehat X_t\}.
	\]
	Then there exists a deterministic function $q:[0,\infty)\to[0,\infty)$ such that for all $t\ge0$,
	\[
		\mathbb{P}(X_t\neq \widehat X_t)=\mathbb{P}(t<T_0)\le q_t\,\mathbb{E}|X_0-\widehat X_0|.
	\]
	Moreover, $q_t$ decays exponentially fast as $t\to\infty$. In particular $q\in L^1(\mathbb{R}_+)$ and there exists a constant
	$C_q>0$, depending only on $(C_b,d)$, such that
	\[
		\int_0^T q_t\,dt\le \int_0^\infty q_t\,dt\le C_q,\qquad \forall\,T\ge0.
	\]
\end{lem}

\begin{proof}
	Let $\Delta_t:=X_t-\widehat X_t$ and $r_t:=|\Delta_t|$. For $t<T_0$, It\^o's formula yields (as in the derivation above)
	\[
		dr_t=\langle e_t,\,b_t(X_t)-b_t(\widehat X_t)\rangle\,dt+2\sqrt2\,d\beta_t,
	\]
	where $e_t=\Delta_t/r_t$ and $\beta_t:=\int_0^t\langle e_s,dW_s\rangle$ is a one-dimensional Brownian motion.
	By Corollary~\ref{cor-Hp-bounded}, $|b_t(\cdot)|\le C_b$, hence
	\[
		\langle e_t,\,b_t(X_t)-b_t(\widehat X_t)\rangle\le |b_t(X_t)|+|b_t(\widehat X_t)|\le 2C_b.
	\]
	Therefore, for $t<T_0$,
	\[
		r_t \le r_0 + 2C_b t + 2\sqrt2\,\beta_t =:\bar r_t,
		\qquad r_0:=|X_0-\widehat X_0|.
	\]
	In particular,
	\[
		\mathbb{P}(t<T_0)
		=
		\mathbb{P}\Big(\inf_{0\le s\le t} r_s>0\Big)
		\le
		\mathbb{P}\Big(\inf_{0\le s\le t} \bar r_s>0\Big).
	\]

	\smallskip
	\noindent\emph{Step 1 (short-time bound).}
	Fix $t_0:=1$. By the reflection principle for one-dimensional Brownian motion with drift, there exists a constant $K=K(C_b)$ such that
	\[
		\mathbb{P}\Big(\inf_{0\le s\le 1} \bar r_s>0\Big)\le K\,r_0,
		\qquad \forall\,r_0\ge0.
	\]
	Hence $\mathbb{P}(1<T_0)\le K\,\mathbb{E}|X_0-\widehat X_0|$.

	\smallskip
	\noindent\emph{Step 2 (exponential decay by restart).}
	Since $X_t,\widehat X_t\in\mathbb{T}^d$, we have the deterministic bound
	\[
		0\le r_t \le \mathrm{diam}(\mathbb{T}^d)=\frac{\sqrt d}{2},\qquad \forall\,t\ge0.
	\]
	Define
	\[
		\eta:=\sup_{r\in(0,\sqrt d/2]}\mathbb{P}\Big(\inf_{0\le s\le 1}(r+2C_b s+2\sqrt2\,\beta_s)>0\Big)\in(0,1).
	\]
	The strict inequality $\eta<1$ follows since for each $r>0$ the drifted Brownian motion hits $0$ within time $1$ with positive probability, and the map $r_0\mapsto \mathbb{P}(\inf_{[0,1]}\bar r_s>0)$ is continuous on $(0,\sqrt d/2]$.

	Let $k\in\mathbb{N}$. Using the strong Markov property at time $k$ and the domination $r_{k+s}\le r_k+2C_b s+2\sqrt2(\beta_{k+s}-\beta_k)$ on $\{k<T_0\}$, we obtain
	\[
		\mathbb{P}(k+1<T_0)\le \eta\,\mathbb{P}(k<T_0),
	\]
	and hence $\mathbb{P}(k<T_0)\le \eta^{k-1}\mathbb{P}(1<T_0)$ for all $k\ge1$. Combining with Step~1 yields
	\[
		\mathbb{P}(t<T_0)\le K\,\eta^{\lfloor t\rfloor-1}\,\mathbb{E}|X_0-\widehat X_0|,\qquad t\ge1,
	\]
	while for $t\in[0,1]$ we use $\mathbb{P}(t<T_0)\le \mathbb{P}(1<T_0)$.
	Thus, the claim holds with
	\[
		q_t:=K\,\mathbf{1}_{[0,1]}(t)+K\,\eta^{\lfloor t\rfloor-1}\mathbf{1}_{(1,\infty)}(t),
	\]
	which decays exponentially. Finally,
	\[
		\int_0^\infty q_t\,dt \le K + K\sum_{k\ge1}\eta^{k-1} = K\Big(1+\frac{1}{1-\eta}\Big)=:C_q,
	\]
	and $C_q$ depends only on $(C_b,d)$.
\end{proof}

\begin{prop}\label{X-lip}
	Assume Assumptions~\ref{ass-potential} and~\ref{ass-cost}. Then
	\begin{align}
		\|\nabla_x u^{T,y}(0,\cdot)\|_{\mathrm{Lip}}
		 & \le \int_0^T 2q_t\bigl(\kappa_1^c + C_\Phi C_\ell\bigr)\,dt,
		\label{eq-x-lip}                                                \\
		\|\nabla_y u^{T,y}(0,\cdot)\|_{\mathrm{Lip}}
		 & \le \int_0^T 2q_t\,C_\Phi C_\ell\,dt.
		\label{eq-y-lip}
	\end{align}
	Here $q_t$ is as in Lemma~\ref{lem-ref-coupling1}.
\end{prop}

\begin{proof}[Proof of \eqref{eq-x-lip}]
	Fix $y\in\mathbb{R}^D$ and $T>0$. For brevity, set
	\[
		h(x,p):=H(x,p)+V(x,y).
	\]
	Let $u^{T,y}$ be the solution of \eqref{eq-T-hjb}. For each $i\in\{1,\dots,d\}$ define
	\[
		u_i(t,x):=\partial_{x_i}u^{T,y}(t,x).
	\]
	Differentiating \eqref{eq-T-hjb} with respect to $x_i$ yields the linear parabolic equation
	\begin{equation}\label{eq:ui-pde}
		\partial_t u_i+\Delta u_i
		+h_p\bigl(x,\nabla_x u^{T,y}(t,x)\bigr)\cdot \nabla_x u_i
		+h_{x_i}\bigl(x,\nabla_x u^{T,y}(t,x)\bigr)=0,
		\qquad
		u_i(T,\cdot)=0,
	\end{equation}
	where $h_{x_i}(x,p)=H_{x_i}(x,p)+\partial_{x_i}V(x,y)$ and $h_p(x,p)=H_p(x,p)$.

	\medskip
	\noindent\emph{Step 1: a Feynman--Kac representation along the optimal flow.}
	Let $Z_t$ solve the controlled diffusion driven by the feedback $h_p$:
	\[
		dZ_t=h_p\bigl(Z_t,\nabla_x u^{T,y}(t,Z_t)\bigr)\,dt+\sqrt2\,dB_t,
		\qquad Z_0=z,
	\]
	where $B$ is a Brownian motion. Applying It\^o's formula to $u_i(t,Z_t)$ and using \eqref{eq:ui-pde} gives
	\[
		du_i(t,Z_t)=-\,h_{x_i}\bigl(Z_t,\nabla_x u^{T,y}(t,Z_t)\bigr)\,dt+\sqrt2\,\nabla_x u_i(t,Z_t)\,dB_t.
	\]
	Taking expectations and using $u_i(T,\cdot)=0$ yields
	\begin{equation}\label{eq:rep-ui}
		u_i(0,z)=\mathbb{E}\Big[\int_0^T h_{x_i}\bigl(Z_t,\nabla_x u^{T,y}(t,Z_t)\bigr)\,dt\Big].
	\end{equation}

	\medskip
	\noindent\emph{Step 2: reflection coupling and difference estimate.}
	Fix $x,\hat x\in\mathbb{T}^d$ and consider the reflection coupling $(X^1_t,\widehat X^1_t)$ defined in \eqref{eq-coupling1}, started from
	$X^1_0=x$ and $\widehat X^1_0=\hat x$, with the same frozen parameter $y$ and horizon $T$.
	Applying \eqref{eq:rep-ui} with $Z=X^1$ and with $Z=\widehat X^1$, and then subtracting, gives
	\[
		u_i(0,x)-u_i(0,\hat x)
		=
		\mathbb{E}\Big[\int_0^T
			\Big(h_{x_i}\bigl(X^1_t,\nabla_x u^{T,y}(t,X^1_t)\bigr)
			-
			h_{x_i}\bigl(\widehat X^1_t,\nabla_x u^{T,y}(t,\widehat X^1_t)\bigr)
			\Big)\,dt\Big].
	\]
	Taking absolute values and using $|\mathbb{E}[\cdot]|\le \mathbb{E}[|\cdot|]$,
	\begin{equation}\label{eq:ui-diff}
		\big|u_i(0,x)-u_i(0,\hat x)\big|
		\le
		\int_0^T
		\mathbb{E}\Big[
		\big|h_{x_i}(X^1_t,\nabla u^{T,y}(t,X^1_t))-h_{x_i}(\widehat X^1_t,\nabla u^{T,y}(t,\widehat X^1_t))\big|
		\Big]\,dt.
	\end{equation}

	\medskip
	\noindent\emph{Step 3: uniform bounds on $H_{x_i}$ and $V_{x_i}$.}
	By Assumption~\ref{ass-cost}.1 and the regularity $c\in C^2$, for every $i$,
	\[
		\sup_{\alpha\in\mathbb{R}^d}\|\partial_{x_i}c(\cdot,\alpha)\|_\infty
		\le
		\sup_{\alpha\in\mathbb{R}^d}\|c(\cdot,\alpha)\|_{\mathrm{Lip}}
		\le \kappa_1^c.
	\]
	Let $\alpha^*(x,p)$ denote the unique minimizer in the definition of $H(x,p)$. By the envelope theorem,
	\[
		\partial_{x_i}H(x,p)=\partial_{x_i}c\bigl(x,\alpha^*(x,p)\bigr),
	\]
	and therefore
	\[
		\|\partial_{x_i}H\|_\infty\le \kappa_1^c.
	\]
	Moreover, $V(x,y)=\nabla\Phi(y)\cdot\ell(x)$, so
	\[
		\|\partial_{x_i}V(\cdot,y)\|_\infty\le C_\Phi C_\ell,
		\qquad\text{uniformly in }y.
	\]
	Consequently, for all $(x,p)$ and $(x',p')$,
	\[
		|h_{x_i}(x,p)-h_{x_i}(x',p')|
		\le |H_{x_i}(x,p)|+|H_{x_i}(x',p')|+|V_{x_i}(x,y)|+|V_{x_i}(x',y)|
		\le 2(\kappa_1^c+C_\Phi C_\ell).
	\]
	Applying this to \eqref{eq:ui-diff} yields
	\[
		\big|u_i(0,x)-u_i(0,\hat x)\big|
		\le
		2(\kappa_1^c+C_\Phi C_\ell)\int_0^T \mathbb{P}(X^1_t\neq \widehat X^1_t)\,dt.
	\]
	By Lemma~\ref{lem-ref-coupling1},
	\[
		\mathbb{P}(X^1_t\neq \widehat X^1_t)\le q_t\,\mathbb{E}|X^1_0-\widehat X^1_0|=q_t\,|x-\hat x|.
	\]
	Therefore,
	\[
		\big|u_i(0,x)-u_i(0,\hat x)\big|
		\le
		\Big(\int_0^T 2q_t(\kappa_1^c+C_\Phi C_\ell)\,dt\Big)\,|x-\hat x|.
	\]
	Since $\nabla_x u^{T,y}(0,\cdot)=(u_1(0,\cdot),\dots,u_d(0,\cdot))$, summing over $i$ gives
	\[
		\|\nabla_x u^{T,y}(0,\cdot)\|_{\mathrm{Lip}}
		\le \int_0^T 2q_t(\kappa_1^c+C_\Phi C_\ell)\,dt,
	\]
	which is exactly \eqref{eq-x-lip}.
\end{proof}

\begin{proof}[Proof of \eqref{eq-y-lip}]
	Fix $y\in\mathbb{R}^D$ and $T>0$. Recall $V(x,y)=\nabla\Phi(y)\cdot \ell(x)$.
	Since $V(\cdot,\cdot)$ is $C^1$ in $y$ and $H$ does not depend on $y$, standard parabolic regularity implies that
	$u^{T,y}$ is $C^1$ in $y$ and we may differentiate \eqref{eq-T-hjb} with respect to $y$.
	Let
	\[
		w(t,x):=\nabla_y u^{T,y}(t,x)\in\mathbb{R}^D.
	\]
	Differentiating \eqref{eq-T-hjb} in $y$ yields the linear vector-valued PDE
	\begin{equation}\label{eq:w-pde}
		\partial_t w+\Delta w+H_p\bigl(x,\nabla_x u^{T,y}(t,x)\bigr)\cdot \nabla_x w+\nabla_y V(x,y)=0,
		\qquad
		w(T,\cdot)=0.
	\end{equation}

	\medskip
	\noindent\emph{Step 1: representation along the optimal flow.}
	Let $Z_t$ solve the diffusion driven by the feedback $H_p$:
	\[
		dZ_t=H_p\bigl(Z_t,\nabla_x u^{T,y}(t,Z_t)\bigr)\,dt+\sqrt2\,dB_t,
		\qquad Z_0=z,
	\]
	with a Brownian motion $B$. Applying It\^o's formula to $w(t,Z_t)$ and using \eqref{eq:w-pde} gives
	\[
		dw(t,Z_t)=-\,\nabla_y V(Z_t,y)\,dt+\sqrt2\,\nabla_x w(t,Z_t)\,dB_t.
	\]
	Taking expectations and using $w(T,\cdot)=0$, we obtain
	\begin{equation}\label{eq:rep-w}
		w(0,z)=\mathbb{E}\Big[\int_0^T \nabla_y V(Z_t,y)\,dt\Big].
	\end{equation}

	\medskip
	\noindent\emph{Step 2: reflection coupling and difference estimate.}
	Fix $x,\hat x\in\mathbb{T}^d$ and consider the reflection coupling $(X^1_t,\widehat X^1_t)$ from \eqref{eq-coupling1},
	started from $X^1_0=x$ and $\widehat X^1_0=\hat x$ (with the same fixed $y$ and horizon $T$).
	Applying \eqref{eq:rep-w} with $Z=X^1$ and $Z=\widehat X^1$, and subtracting, yields
	\[
		w(0,x)-w(0,\hat x)
		=
		\mathbb{E}\Big[\int_0^T
			\big(\nabla_y V(X^1_t,y)-\nabla_y V(\widehat X^1_t,y)\big)\,dt\Big].
	\]
	Therefore,
	\begin{equation}\label{eq:w-diff}
		\|w(0,x)-w(0,\hat x)\|
		\le
		\int_0^T
		\mathbb{E}\Big[\big\|\nabla_y V(X^1_t,y)-\nabla_y V(\widehat X^1_t,y)\big\|\Big]\,dt.
	\end{equation}
	We have
	\[
		\nabla_y V(x,y)=\nabla^2\Phi(y)\,\ell(x),
	\]
	hence, by Assumption~\ref{ass-potential} and the bound $\|\ell\|_\infty\le C_\ell$,
	\[
		\|\nabla_y V(\cdot,y)\|_\infty\le C_\Phi C_\ell,
		\qquad\text{uniformly in }y.
	\]
	Consequently,
	\[
		\big\|\nabla_y V(X^1_t,y)-\nabla_y V(\widehat X^1_t,y)\big\|
		\le 2C_\Phi C_\ell\,\mathbf 1_{\{X^1_t\neq \widehat X^1_t\}}.
	\]
	Inserting this into \eqref{eq:w-diff} and using Lemma~\ref{lem-ref-coupling1} gives
	\[
		\|w(0,x)-w(0,\hat x)\|
		\le
		2C_\Phi C_\ell\int_0^T \mathbb{P}(X^1_t\neq \widehat X^1_t)\,dt
		\le
		\Big(\int_0^T 2C_\Phi C_\ell\,q_t\,dt\Big)\,|x-\hat x|.
	\]
	Therefore,
	\[
		\|\nabla_y u^{T,y}(0,\cdot)\|_{\mathrm{Lip}}
		\le
		\int_0^T 2C_\Phi C_\ell\,q_t\,dt.
	\]
	With the symmetry of mixed derivatives, this implies
	\[
		\|\nabla_x u^{T,y}(0,\cdot)\|_{\mathrm{Lip}_y}
		=
		\|\nabla_y u^{T,y}(0,\cdot)\|_{\mathrm{Lip}_x}
		\le
		\int_0^T 2C_\Phi C_\ell\,q_t\,dt.
	\]
\end{proof}

\paragraph{Lipschitz continuity of the drift.}
We now deduce Lipschitz bounds for $\nabla u^y$ and for $H_p(x,\nabla u^y(x))$.
By the turnpike estimate (see~\cite[Theorem~2.(1).(b)]{CecchinConfortiDurmusEichinger2024}),
\[
	\|\nabla_x u^{T,y}(t,\cdot)-\nabla_x u^{y}(\cdot)\|_\infty
	\le C e^{-c(T-t)},
\]
uniformly in $(T,t,y)$ (after fixing a normalization of $u^y$). Taking $t=0$ and $T\to\infty$ implies convergence of $\nabla_x u^{T,y}(0,\cdot)$ to $\nabla_x u^y(\cdot)$, and therefore from Proposition~\ref{X-lip},
\[
	\|\nabla_x u^{y}\|_{\Lip_x}
	\le \sup_{T>0}\|\nabla_x u^{T,y}(0,\cdot)\|_{\Lip_x}
	\le 2C_q(\kappa_1^c+C_\Phi C_\ell),
	\qquad
	\|\nabla_x u^{y}\|_{\Lip_y}
	\le 2C_q C_\Phi C_\ell.
\]

Finally, using the implicit function theorem for the minimizer $\alpha^*(x,p)=H_p(x,p)$, we have
\begin{align*}
	\partial_p H_p(x,p) & =\alpha^*_p(x,p)=-(c_{\alpha\alpha}(x,\alpha^*))^{-1},            \\
	\partial_x H_p(x,p) & =\alpha^*_x(x,p)=-(c_{\alpha\alpha})^{-1}c_{x\alpha}(x,\alpha^*),
\end{align*}
so by Assumption~\ref{ass-cost} and the bound $|\alpha^*|\le C_b$ from Corollary~\ref{cor-Hp-bounded},
\[
	\|H_p(\cdot,p)\|_{\Lip_x}\le \kappa_2^c\kappa_3^c(1+C_b)^{\theta_c},
	\qquad
	\|H_p(x,\cdot)\|_{\Lip_p}\le \kappa_2^c.
\]
Combining these yields a joint Lipschitz bound, uniformly in $y$:
\[
	\bigl|H_p(x,\nabla u^y(x))-H_p(x',\nabla u^{y'}(x'))\bigr|
	\le C_L\bigl(|x-x'|+|y-y'|\bigr),
\]
with
\[
	C_L
	\le
	\kappa_2^c\kappa_3^c(1+C_b)^{\theta_c}
	+\kappa_2^c\Bigl(2C_q(\kappa_1^c+C_\Phi C_\ell)+2C_q C_\Phi C_\ell\Bigr),
\]
which is precisely the estimate needed for Lemma~\ref{lem-regularity}.

\subsection{Proof of Theorem~\ref{thm-main}}\label{proof-thm-2}

We keep the notation of Section~\ref{section-main-result}. In particular, $\rho_\lambda$ denotes the unique invariant law of the SFP system~\eqref{eq-SFP-XY}, and for each fixed $y\in\mathbb{R}^D$, $\widehat\rho_y$ denotes the invariant measure of the frozen dynamics~\eqref{eq-frozen}.

\medskip
We begin by quantifying the distance between the stationary distribution $\rho_\lambda$ of the SFP and the conditional equilibria $\widehat{\rho}_y$ using an original coupling approach. We then use the Lasry–Lions divergence as the main bridge to prove that the Wasserstein distance between the conditional equilibria $\widehat{\rho}_y$ and the MFG Nash equilibrium $m^*$ is controlled by the distance between $\rho_\lambda$ and $\widehat{\rho}_y$. Finally we will complete the error estimate between $\rho_\lambda$ and $m^*$ by invoking the triangle inequality.

\medskip
In the subsequent proof, we will repeatedly make use of the following stability  result via reflection coupling.

\begin{lem}[Stability by reflection coupling]\label{lem-reflection-coupling}
	Consider the following two diffusion processes:
	\[
		dZ_t = b(t,Z_t)\,dt + \sigma\, dW_t,
		\qquad
		dZ'_t = (b+\delta b)(t,Z'_t)\,dt + \sigma\, dW_t,
	\]
	and denote their marginal distributions by
	\[
		p_t := {\rm Law}(Z_t),
		\qquad
		p'_t :={\rm Law}(Z'_t).
	\]
	Assume that the drifts $b$ and $\delta b$ satisfy:
	\begin{enumerate}
		\item[{\rm (i)}] $b$ and $\delta b$ are Lipschitz in $z$, i.e., there exists a constant $L > 0$ such that
			\[
				|b(t,z) - b(t,z')| + |\delta b(t,z) - \delta b(t,z')|
				\le L |z - z'|,
			\]
			for all $t \in [0,T]$, $z,z' \in \mathbb{R}^d$.

		\item[{\rm (ii)}] There exists a continuous function
			$\kappa : (0,\infty) \to \mathbb{R}$
			such that
			\[
				\limsup_{r \to \infty} \kappa(r) < 0,
				\qquad
				\int_0^1 r \kappa^+(r)\,dr < \infty,
			\]
			and
			\[
				(z-z')\cdot (b(t,z)-b(t,z'))
				\le \kappa(|z-z'|)\,|z-z'|^2,
			\]
			for all $t \in [0,T]$, $z,z' \in \mathbb{R}^d$.
	\end{enumerate}
	Then we have
	\begin{equation}
		\cw_1(p_t, p'_t)
		\le \mathcal{C} e^{-\mathcal{R} \sigma^2 t}
		\left(
		\cw_1(p_0, p'_0)
		+ \int_0^t e^{\mathcal{R} \sigma^2 s}
		\mathbb{E}\big[ |\delta b(s,Z'_s)| \big]\,ds
		\right),
		\quad \text{for all } t \ge 0,
	\end{equation}
	where the constants $\mathcal{C}$ and $\mathcal{R}$ depend only on the function $\kappa(\cdot)/\sigma^2$.
\end{lem}

For the detailed proof of this lemma, we refer the reader to \cite[Theorem~A.7]{ClaisseConfortiRenWang2023}, and omit it here.

\subsubsection{A first comparison: \texorpdfstring{$\rho_\lambda$}{rho lambda} versus \texorpdfstring{$\widehat\rho_\cdot\otimes\rho_\lambda^Y$}{widehat rho cdot otimes rho lambda Y}}

Unlike the MFC case, the drift term in the dynamics of $X_t$ is no longer of gradient-field form, and therefore the entropy-based estimates used in \cite[Section~6.2]{DuRenSuciuWang2024} are no longer applicable. To overcome this difficulty, we use reflection coupling to obtain the comparison estimate of order $O(\lambda)$ in the stationary regime.

\begin{prop}\label{prop-rho-rhohat}
	Let Assumptions~\ref{ass-potential} and~\ref{ass-cost} hold. Then
	\[
		\mathcal{W}_1\!\left(\rho_\lambda,\ \widehat\rho_{\cdot}\otimes\rho_\lambda^{Y}\right)
		\;\le\; C_{\mathrm{fr}}\,\lambda,
	\]
	where $C_{\mathrm{fr}}>0$ depends only on $(C_\ell,C_b,C_L,d)$ (in particular it is independent of $\lambda$).
\end{prop}

\begin{proof}
	We compare SFP to a frozen-$Y$ dynamics over a finite time and then use exponential mixing for the frozen-$Y$ semigroup.

	\medskip
	\noindent\emph{Step 1: freeze $Y$ and estimate the synchronous error.}
	Let $(X_t,Y_t)$ solve \eqref{eq-SFP-XY} with $\Law(X_0,Y_0)=\rho_\lambda$. Define $(\widetilde X_t,\widetilde Y_t)$ by
	\[
		d\widetilde X_t
		=
		H_p\!\bigl(\widetilde X_t,\nabla u^{\widetilde Y_t}(\widetilde X_t)\bigr)\,dt+\sqrt2\,dW_t,
		\qquad
		\widetilde Y_t\equiv Y_0,
		\qquad
		(\widetilde X_0,\widetilde Y_0)=(X_0,Y_0),
	\]
	driven by the same Brownian motion $W_t$.
	Set $\delta X_t:=X_t-\widetilde X_t$ and $\delta Y_t:=Y_t-\widetilde Y_t=Y_t-Y_0$.
	By Lemma~\ref{lem-regularity} and synchronous coupling,
	\begin{equation}\label{eq:fr-dX}
		\frac{d}{dt}\,\mathbb{E}\!\left[|\delta X_t|\right]
		\le
		C_L\,\mathbb{E}\!\left[|\delta X_t|+|\delta Y_t|\right].
	\end{equation}
	Moreover, since $dY_t=\lambda(\ell(X_t)-Y_t)\,dt$ and $\|\ell\|_\infty\le C_\ell$, while the $Y$-marginal of $\rho_\lambda$
	is supported in the convex hull of $\ell(\mathbb{T}^d)$ (hence $|Y_t|\le C_\ell$ a.s.), we have
	\begin{equation}\label{eq:fr-dY}
		|\delta Y_t|
		=
		|Y_t-Y_0|
		\le
		\lambda\int_0^t \bigl(|\ell(X_s)|+|Y_s|\bigr)\,ds
		\le
		2\lambda C_\ell\,t.
	\end{equation}
	Combining \eqref{eq:fr-dX}--\eqref{eq:fr-dY} and using $\delta X_0=0$ yields
	\[
		\mathbb{E}\!\left[|\delta X_t|\right]
		\le
		C_L\int_0^t e^{C_L(t-s)}\,\mathbb{E}\!\left[|\delta Y_s|\right]\,ds
		\le
		2\lambda C_\ell C_L\int_0^t s\,e^{C_L(t-s)}\,ds
		=
		2\lambda C_\ell\left(\frac{e^{C_L t}-1}{C_L}-t\right).
	\]
	Therefore, for all $t\ge0$,
	\begin{equation}\label{eq:fr-W1}
		\mathcal{W}_1\!\left(\Law(X_t,Y_t),\ \Law(\widetilde X_t,\widetilde Y_t)\right)
		\le
		\mathbb{E}\!\left[|\delta X_t|+|\delta Y_t|\right]
		\le
		\frac{2\lambda C_\ell}{C_L}\,e^{C_Lt}.
	\end{equation}

	\medskip
	\noindent\emph{Step 2: contract the frozen-$Y$ semigroup toward $\widehat\rho_\cdot\otimes\rho_\lambda^Y$.}
	Define $(\widehat X_t,\widehat Y_t)$ by
	\[
		d\widehat X_t
		=
		H_p\!\bigl(\widehat X_t,\nabla u^{\widehat Y_t}(\widehat X_t)\bigr)\,dt+\sqrt2\,dW_t,
		\qquad
		\widehat Y_t\equiv \widehat Y_0,
	\]
	where $\Law(\widehat Y_0)=\rho_\lambda^Y$ and, conditionally on $\widehat Y_0=y$, $\Law(\widehat X_0\mid \widehat Y_0=y)=\widehat\rho_y$.
	Then $\Law(\widehat X_t,\widehat Y_t)=\widehat\rho_\cdot\otimes\rho_\lambda^Y$ for all $t\ge0$.

	Since $x\mapsto H_p(x,\nabla u^y(x))$ is uniformly bounded and Lipschitz in $x$, uniformly over $y$ (Lemma~\ref{lem-regularity}),
	the reflection coupling estimate as in Lemma \ref{lem-reflection-coupling} yields constants $\mathcal{C},\mathcal{R}>0$ depending only on $(C_b,C_L,d)$ such that
	\begin{equation}\label{eq:fr-contract}
		\mathcal{W}_1\!\left(\Law(\widetilde X_t,\widetilde Y_t),\ \widehat\rho_\cdot\otimes\rho_\lambda^Y\right)
		\le
		\mathcal{C}e^{-\mathcal{R}t}\,
		\mathcal{W}_1\!\left(\Law(\widetilde X_0,\widetilde Y_0),\ \widehat\rho_\cdot\otimes\rho_\lambda^Y\right).
	\end{equation}
	Because $\Law(\widetilde X_0,\widetilde Y_0)=\rho_\lambda$, the right-hand side equals
	$\mathcal{C}e^{-\mathcal{R}t}\,\mathcal{W}_1\!\left(\rho_\lambda,\widehat\rho_\cdot\otimes\rho_\lambda^Y\right)$.

	\medskip
	\noindent\emph{Step 3: conclude by a fixed-time choice.}
	By the triangle inequality and \eqref{eq:fr-W1}--\eqref{eq:fr-contract},
	\[
		\mathcal{W}_1\!\left(\rho_\lambda,\widehat\rho_\cdot\otimes\rho_\lambda^Y\right)
		\le
		\frac{2\lambda C_\ell}{C_L}\,e^{C_Lt}
		+
		\mathcal{C}e^{-\mathcal{R}t}\,
		\mathcal{W}_1\!\left(\rho_\lambda,\widehat\rho_\cdot\otimes\rho_\lambda^Y\right).
	\]
	Choose $t=t_0:=\mathcal{R}^{-1}\log(2\mathcal{C})$, so that $\mathcal{C}e^{-\mathcal{R}t_0}=\tfrac12$. Then
	\[
		\mathcal{W}_1\!\left(\rho_\lambda,\widehat\rho_\cdot\otimes\rho_\lambda^Y\right)
		\le
		\frac{4\lambda C_\ell}{C_L}\,e^{C_Lt_0}
		=: C_{\mathrm{fr}}\,\lambda,
	\]
	which proves the claim.
\end{proof}

\subsubsection{The auxiliary measure \texorpdfstring{$\mu$}{mu} and its properties}

We now introduce a probability measure $\mu=\mu_\lambda$ on $\mathbb{T}^d\times\mathbb{R}^D$ built from the invariant law $\rho_\lambda$.
Note that in the original SFP formulation \eqref{eq-SFP} also admits a unique invariant measure $\bar\rho_\lambda\in \mathcal{P}(\mathbb{T}^d\times\mathcal{P}(\mathbb{T}^d))$. Let $(X_\lambda, m_\lambda)\sim \bar\rho_\lambda$ and  $Y_\lambda:=\langle \ell,m_\lambda\rangle$.
For any bounded measurable test function $f:\mathbb{T}^d\times\mathbb{R}^D\to\mathbb{R}$, define
\begin{equation}\label{eq:def-mu}
	\langle f,\mu\rangle
	:=
	\int_{\mathbb{T}^d\times\mathbb{R}^D} f(x,y)\,\mu(dx\,dy)
	:=
	\mathbb{E}\!\left[\int_{\mathbb{T}^d} f(x,Y_\lambda)\,m_\lambda(dx)\right].
\end{equation}
By taking $f$ depending only on $y$, one immediately checks that the $Y$-marginal laws $\mu^Y=\rho_\lambda^Y$ coincide.

\begin{prop}\label{prop:mu-properties}
	Assume Assumption~\ref{ass-cost}. For every $\Psi\in C^2(\mathbb{R}^D;\mathbb{R})$ with bounded Hessian,
	\begin{equation}\label{eq:mu-rho-orthogonality}
		\iint_{\mathbb{T}^d\times\mathbb{R}^D} \nabla\Psi(y)\cdot \ell(x)\,(\mu-\rho_\lambda)(dx\,dy)=0.
	\end{equation}
	In particular, writing $\mu^X$ and $\rho_\lambda^X$ for the first marginals,
	\[
		\int_{\mathbb{T}^d}\ell(x)\,\bigl(\mu^X-\rho_\lambda^X\bigr)(dx)=0.
	\]
	Moreover, if $\mu^{X|Y}(\cdot\mid\cdot)$ denotes a regular conditional law of $X$ given $Y$ under $\mu$, then for $\rho_\lambda^Y$-a.e.\ $y$,
	\begin{equation}\label{eq:conditional-expectation}
		\langle \ell,\mu^{X|Y}(\cdot\mid y)\rangle = y.
	\end{equation}
\end{prop}

The proof follows the proof in \cite[Proposition~8]{DuRenSuciuWang2024} (we include details in Appendix~\ref{app-mu} for completeness).

\subsubsection{Proof of the rest of Theorem~\ref{thm-main}}

\begin{proof}
	The first result in Theorem~\ref{thm-main}  has already been proved in Proposition~\ref{prop-rho-rhohat}. We split the rest of the proof into three steps.

	\medskip
	\noindent\emph{Step 1. Lasry--Lions divergence controls an $L^2$-gap of gradients.}
	Recall that $H_p(x,p)=\alpha^*(x,p)$, where $\alpha^*$ is defined by the first-order condition
	$p+c_\alpha(x,\alpha)=0$. By the implicit function theorem and Assumption~\ref{ass-cost}(ii),
	\[
		H_{pp}(x,p)
		=
		-\bigl(c_{\alpha\alpha}(x,H_p(x,p))\bigr)^{-1}
		\preceq -\frac{1}{\kappa_2^c}\,I_d,
	\]
	hence $p\mapsto H(x,p)$ is $\frac{1}{\kappa_2^c}$-strongly concave. Therefore, for all $p,p'$,
	\[
		H(x,p)-H(x,p')-H_p(x,p)\cdot(p-p')\ \ge\ \frac{1}{2\kappa_2^c}\,\|p-p'\|^2,
	\]
	and symmetrically with $p,p'$ exchanged. Integrating with $p=\nabla u$ and $p'=\nabla u'$ against $m$ and $m'$ and summing yields
	\begin{equation}\label{eq-LL-L2-mainproof}
		D_{\mathrm{LL}}(\nabla u,\nabla u')
		\ \ge\
		\frac{1}{2\kappa_2^c}\int_{\mathbb{T}^d}\|\nabla u-\nabla u'\|^2\,d(m+m').
	\end{equation}

	\medskip
	\noindent\emph{Step 2. Control $\int D_{\mathrm{LL}}(\nabla u^*,\nabla u^y)\,\rho_\lambda^Y(dy)$ using $\mu$.}
	Let $y^*:=\langle \ell,m^*\rangle$. The ergodic MFG equilibrium $(u^*,m^*)$ solves
	\begin{align*}
		\Delta u^* + H(x,\nabla u^*) + V(y^*,x)                        & = \delta^*, \\
		\Delta m^* - \nabla\!\cdot\!\bigl(m^*\,H_p(x,\nabla u^*)\bigr) & = 0,
	\end{align*}
	and for each fixed $y\in\mathbb{R}^D$, $(u^y,\widehat\rho_y)$ solves
	\begin{align*}
		\Delta u^y + H(x,\nabla u^y) + V(y,x)                                                & = \widehat\delta_y, \\
		\Delta \widehat\rho_y - \nabla\!\cdot\!\bigl(\widehat\rho_y\,H_p(x,\nabla u^y)\bigr) & = 0.
	\end{align*}
	The standard Lasry--Lions computation (subtract the HJB equations and test by $\widehat\rho_y-m^*$, subtract the FP equations and test by $u^*-u^y$, then integrate by parts) gives
	\begin{equation}\label{eq:LL-id-mainproof}
		D_{\mathrm{LL}}(\nabla u^*,\nabla u^y)
		=
		-\int_{\mathbb{T}^d}\bigl(V(y,x)-V(y^*,x)\bigr)\,\bigl(m^*-\widehat\rho_y\bigr)(dx).
	\end{equation}
	Integrating \eqref{eq:LL-id-mainproof} with respect to $\rho_\lambda^Y(dy)$ yields
	\begin{equation}\label{eq:LL-int-compact-mainproof}
		\int_{\mathbb{R}^D}D_{\mathrm{LL}}(\nabla u^*,\nabla u^y)\,\rho_\lambda^Y(dy)
		=
		\iint \bigl(V(y,x)-V(y^*,x)\bigr)\,
		d\Bigl(\widehat\rho_\cdot\otimes\rho_\lambda^Y - m^*\otimes\rho_\lambda^Y\Bigr)(x,y).
	\end{equation}
	Decompose
	\[
		\widehat\rho_\cdot\otimes\rho_\lambda^Y - m^*\otimes\rho_\lambda^Y
		=
		\Bigl(\widehat\rho_\cdot\otimes\rho_\lambda^Y-\rho_\lambda\Bigr)
		+
		\Bigl(\rho_\lambda-m^*\otimes\rho_\lambda^Y\Bigr),
	\]
	and accordingly write the right-hand side of \eqref{eq:LL-int-compact-mainproof} as $I_1+I_2$, where
	\begin{align*}
		I_1 & :=\iint \bigl(V(y,x)-V(y^*,x)\bigr)\,d\Bigl(\widehat\rho_\cdot\otimes\rho_\lambda^Y-\rho_\lambda\Bigr)(x,y), \\
		I_2 & :=\iint \bigl(V(y,x)-V(y^*,x)\bigr)\,d\Bigl(\rho_\lambda-m^*\otimes\rho_\lambda^Y\Bigr)(x,y).
	\end{align*}

	\smallskip
	\noindent\emph{Claim: $I_2\ge0$.}
	This is where Proposition~\ref{prop:mu-properties} is used. Recall that $V(y,x)=\nabla\Phi(y)\cdot \ell(x)$ and $\Phi$ is convex.
	Let $\mu$ be the auxiliary measure defined by \eqref{eq:def-mu}. Taking $\Psi=\Phi$ in Proposition~\ref{prop:mu-properties} gives
	\[
		\iint \nabla\Phi(y)\cdot \ell(x)\,(\mu-\rho_\lambda)(dx\,dy)=0.
	\]
	Since $\mu^Y=\rho_\lambda^Y$ and $\int \ell\,d\mu^X=\int \ell\,d\rho_\lambda^X$ (again Proposition~\ref{prop:mu-properties}),
	the same identity holds after subtracting the constant vector $\nabla\Phi(y^*)$, hence
	\begin{equation}\label{eq:rho-to-mu-mainproof}
		\iint \bigl(\nabla\Phi(y)-\nabla\Phi(y^*)\bigr)\cdot \ell(x)\,\rho_\lambda(dx\,dy)
		=
		\iint \bigl(\nabla\Phi(y)-\nabla\Phi(y^*)\bigr)\cdot \ell(x)\,\mu(dx\,dy).
	\end{equation}
	Therefore,
	\[
		I_2
		=
		\iint \bigl(\nabla\Phi(y)-\nabla\Phi(y^*)\bigr)\cdot \ell(x)\,(\mu-m^*\otimes\rho_\lambda^Y)(dx\,dy).
	\]
	Disintegrate $\mu(dx\,dy)=\mu^{X|Y}(dx\mid y)\,\rho_\lambda^Y(dy)$. By \eqref{eq:conditional-expectation},
	$\langle \ell,\mu^{X|Y}(\cdot\mid y)\rangle = y$ for $\rho_\lambda^Y$-a.e.\ $y$. Hence
	\begin{align*}
		I_2
		 & =
		\int_{\mathbb{R}^D}
		\bigl(\nabla\Phi(y)-\nabla\Phi(y^*)\bigr)\cdot
		\Bigl(\langle \ell,\mu^{X|Y}(\cdot\mid y)\rangle-\langle \ell,m^*\rangle\Bigr)\,\rho_\lambda^Y(dy) \\
		 & =
		\int_{\mathbb{R}^D}
		\bigl(\nabla\Phi(y)-\nabla\Phi(y^*)\bigr)\cdot (y-y^*)\,\rho_\lambda^Y(dy).
	\end{align*}
	Since $\Phi$ is convex, $\nabla\Phi$ is monotone and thus the integrand is nonnegative; therefore $I_2\ge0$.

	\smallskip
	As a consequence,
	\begin{equation}\label{eq:LL-bound-by-I1-mainproof}
		\int_{\mathbb{R}^D}D_{\mathrm{LL}}(\nabla u^*,\nabla u^y)\,\rho_\lambda^Y(dy)
		\le I_1.
	\end{equation}
	Finally, on $\mathbb{T}^d\times\mathrm{supp}(\rho_\lambda^Y)$ the map $(x,y)\mapsto V(y,x)$ is Lipschitz with constant $2C_\Phi C_\ell$
	(because $\|\ell\|_\infty+\|\nabla \ell\|_\infty\le C_\ell$ and $\|\nabla\Phi\|+ \|\nabla^2\Phi\|\le C_\Phi$ on the relevant bounded set).
	Hence
	\begin{equation}\label{eq:LL-to-W1-mainproof}
		I_1
		\le
		2C_\Phi C_\ell\,\mathcal{W}_1\!\left(\rho_\lambda,\widehat\rho_\cdot\otimes\rho_\lambda^Y\right).
	\end{equation}
	Combining \eqref{eq:LL-bound-by-I1-mainproof} and \eqref{eq:LL-to-W1-mainproof} yields
	\begin{equation}\label{eq:LL-final-mainproof}
		\int_{\mathbb{R}^D}D_{\mathrm{LL}}(\nabla u^*,\nabla u^y)\,\rho_\lambda^Y(dy)
		\le
		2C_\Phi C_\ell\,\mathcal{W}_1\!\left(\rho_\lambda,\widehat\rho_\cdot\otimes\rho_\lambda^Y\right).
	\end{equation}
	Using Proposition~\ref{prop-rho-rhohat} (the $O(\lambda)$ freezing bound) gives the stated $O(\lambda)$ control of the averaged divergence.

	\medskip
	\noindent\emph{Step 3. Conclude the $\mathcal{W}_1$ estimate to the Nash equilibrium.}
	Fix $y$ and consider the diffusions
	\[
		dX_t^* = H_p(X_t^*,\nabla u^*(X_t^*))\,dt+\sqrt2\,dW_t,
		\qquad
		dX_t^y = H_p(X_t^y,\nabla u^y(X_t^y))\,dt+\sqrt2\,d\widehat W_t,
	\]
	with $\Law(X_0^*)=m^*$ and $\Law(X_0^y)=\widehat\rho_y$. By Lemma~\ref{lem-regularity}, the drifts are uniformly bounded and Lipschitz in $x$
	with constants controlled by $(C_b,C_L)$, uniformly over $y$.
	Applying the stability estimate in Lemma~\ref{lem-reflection-coupling} yields a constant
	\[
		\Gamma:=\frac{\kappa_2^c\kappa_3^c}{\mathcal{R}}(1+C_b)^{\theta_c}\,\mathcal{C},
	\]
	such that
	\begin{equation}\label{eq:stab-W1-mainproof}
		\mathcal{W}_1(m^*,\widehat\rho_y)
		\le
		\Gamma\int_{\mathbb{T}^d}\bigl|\nabla u^*(x)-\nabla u^y(x)\bigr|\,m^*(dx).
	\end{equation}
	Integrating \eqref{eq:stab-W1-mainproof} against $\rho_\lambda^Y(dy)$ and using Cauchy--Schwarz,
	\begin{align}
		\int_{\mathbb{R}^D} \mathcal{W}_1(\widehat\rho_y,m^*)\,\rho_\lambda^Y(dy)
		 & \le
		\Gamma\iint \bigl|\nabla u^*-\nabla u^y\bigr|\,m^*(dx)\,\rho_\lambda^Y(dy)
		\nonumber \\
		 & \le
		\Gamma\sqrt{\iint \bigl|\nabla u^*-\nabla u^y\bigr|^2\,d(m^*+\widehat\rho_y)\,\rho_\lambda^Y(dy)}.
		\label{eq:W1-L2-mainproof}
	\end{align}
	By \eqref{eq-LL-L2-mainproof},
	\[
		\iint \bigl|\nabla u^*-\nabla u^y\bigr|^2\,d(m^*+\widehat\rho_y)\,\rho_\lambda^Y(dy)
		\le
		4\kappa_2^c \int_{\mathbb{R}^D} D_{\mathrm{LL}}(\nabla u^*,\nabla u^y)\,\rho_\lambda^Y(dy).
	\]
	Combining with \eqref{eq:LL-final-mainproof} gives
	\[
		\int_{\mathbb{R}^D} \mathcal{W}_1(\widehat\rho_y,m^*)\,\rho_\lambda^Y(dy)
		\le
		\Gamma\sqrt{8\kappa_2^c C_\Phi C_\ell\,
			\mathcal{W}_1\!\left(\rho_\lambda,\widehat\rho_\cdot\otimes\rho_\lambda^Y\right)}.
	\]
	Using Proposition~\ref{prop-rho-rhohat}, we obtain the $O(\lambda^{1/2})$ estimate.

	Finally,
	\[
		\mathcal{W}_1\!\bigl(\widehat\rho_\cdot\otimes\rho_\lambda^Y,\ m^*\otimes\rho_\lambda^Y\bigr)
		\le
		\int_{\mathbb{R}^D} \mathcal{W}_1(\widehat\rho_y,m^*)\,\rho_\lambda^Y(dy),
	\]
	and by the triangle inequality,
	\[
		\mathcal{W}_1\!\bigl(\rho_\lambda,\ m^*\otimes\rho_\lambda^Y\bigr)
		\le
		\mathcal{W}_1\!\bigl(\rho_\lambda,\widehat\rho_\cdot\otimes\rho_\lambda^Y\bigr)
		+
		\mathcal{W}_1\!\bigl(\widehat\rho_\cdot\otimes\rho_\lambda^Y,\ m^*\otimes\rho_\lambda^Y\bigr),
	\]
	which yields exactly the claimed bound
	\[
		\mathcal{W}_1\!\bigl(\rho_\lambda,\ m^*\otimes\rho_\lambda^Y\bigr)
		\;\le\;
		C_{\mathrm{fr}}\,\lambda
		+
		\Gamma\sqrt{8\kappa_2^c\,C_\Phi C_\ell\,C_{\mathrm{fr}}}\ \lambda^{1/2}.
	\]
\end{proof}

\begin{appendix}

	\section{Proof of Theorem~\ref{thm-SFP-converge}}\label{app-thm}

	\begin{proof}
		Recall that the SFP dynamics on $\mathbb{T}^d\times\mathbb{R}^D$ reads
		\[
			dX_t=b(X_t,Y_t)\,dt+\sqrt{2}\,dB_t,
			\qquad
			dY_t=\lambda\big(\ell(X_t)-Y_t\big)\,dt,
		\]
		where $b(x,y):=H_p\bigl(x,\nabla u^{y}(x)\bigr)$.
		Let $(X_t,Y_t)$ and $(X'_t,Y'_t)$ be two solutions started from arbitrary initial laws.
		Throughout the proof we use the Lipschitz estimate of Lemma~\ref{lem-regularity}:
		\begin{equation}\label{eq:b-Lip-proof1}
			\|b(x,y)-b(x',y')\|\le C_L\big(\|x-x'\|+\|y-y'\|\big),
			\qquad
			\sup_{x,y}\|b(x,y)\|\le C_b.
		\end{equation}

		\medskip
		\noindent\emph{Step 1: a mixed reflection/synchronous coupling.}
		Fix $n\in\mathbb{N}$ and choose a Lipschitz function $r_n:\mathbb{T}^d\times\mathbb{T}^d\to[0,1]$ such that
		$s_n:=\sqrt{1-r_n^2}$ is Lipschitz and
		\[
			r_n(x,x')=
			\begin{cases}
				1, & \|x-x'\|\ge 2n^{-1}, \\
				0, & \|x-x'\|\le n^{-1}.
			\end{cases}
		\]
		Let $(B^{(1)}_t)_{t\ge0}$ and $(B^{(2)}_t)_{t\ge0}$ be independent $d$-dimensional Brownian motions and set
		\[
			e_t:=
			\begin{cases}
				\dfrac{X_t-X'_t}{\|X_t-X'_t\|}, & X_t\neq X'_t, \\[1ex]
				(1,0,\dots,0)^\top,             & X_t=X'_t.
			\end{cases}
		\]
		Define the coupled noises by
		\[
			\begin{aligned}
				dB_t  & = r_n(X_t,X'_t)\,dB^{(1)}_t+s_n(X_t,X'_t)\,dB^{(2)}_t,                     \\
				dB'_t & = r_n(X_t,X'_t)\,(I_d-2e_te_t^\top)\,dB^{(1)}_t+s_n(X_t,X'_t)\,dB^{(2)}_t.
			\end{aligned}
		\]
		This is the standard Eberle coupling: it is reflection when $\|X_t-X'_t\|$ is not too small and becomes synchronous at very small distances.

		\medskip
		\noindent\emph{Step 2: a distance process.}
		Set $\Delta X_t:=X_t-X'_t$ and $\Delta Y_t:=Y_t-Y'_t$, and write $R^X_t:=\|\Delta X_t\|$, $R^Y_t:=\|\Delta Y_t\|$.
		Subtracting the $X$-equations gives
		\[
			d\Delta X_t=\big(b(X_t,Y_t)-b(X'_t,Y'_t)\big)\,dt+2\,r_n(X_t,X'_t)\,e_t\,dW_t,
		\]
		where
		\(
		W_t:=\int_0^t e_s^\top\,dB^{(1)}_s
		\)
		is a one-dimensional Brownian motion.
		By It\^o's formula for $R^X_t=\|\Delta X_t\|$ (the second-order term vanishes as usual for reflection coupling), we obtain
		\begin{equation}\label{eq:dRX}
			dR^X_t
			\le
			\big\langle e_t,\,b(X_t,Y_t)-b(X'_t,Y'_t)\big\rangle\,dt
			+2\,r_n(X_t,X'_t)\,dW_t.
		\end{equation}
		Using \eqref{eq:b-Lip-proof1} yields
		\begin{equation}\label{eq:dRX-Lip}
			dR^X_t \le C_L(R^X_t+R^Y_t)\,dt+2\,r_n(X_t,X'_t)\,dW_t.
		\end{equation}
		For $R^Y_t$ we use that $d\Delta Y_t=\lambda(\ell(X_t)-\ell(X'_t)-\Delta Y_t)\,dt$. Since $\ell$ is $C_\ell$-Lipschitz,
		\begin{equation}\label{eq:dRY}
			dR^Y_t\le -\lambda R^Y_t\,dt+\lambda C_\ell R^X_t\,dt.
		\end{equation}

		Define the weighted distance
		\[
			R_t:=R^X_t+\frac{2C_L}{\lambda}R^Y_t.
		\]
		Combining \eqref{eq:dRX-Lip}--\eqref{eq:dRY} gives the semimartingale decomposition
		\begin{equation}\label{eq:dR}
			dR_t \le -\Big(\lambda_0-\frac{M}{R_t}\Big)R_t\,dt + 2\,r_n(X_t,X'_t)\,dW_t,
			\qquad t<\tau_0,
		\end{equation}
		where $\tau_0:=\inf\{t:\ R_t=0\}$ and where we set
		\[
			\lambda_0:=\min\Big\{1,\frac{\lambda}{2}\Big\},
			\qquad
			M:=\frac{\sqrt d}{2}\,\big(1+C_L+2C_LC_\ell\big).
		\]
		Here we used $R^X_t\le \sqrt d/2$ on $\mathbb{T}^d$, hence the positive terms involving $R^X_t$ can be bounded by $M/R_t$ after dividing by $R_t\ge R^X_t$.

		\medskip
		\noindent\emph{Step 3: a concave Lyapunov function.}
		Let $K(r):=\lambda_0-\frac{M}{r}$ for $r>0$. Following \cite[Lemma~5.1]{du2023sequential}, choose a $C^2$ concave function
		$f:[0,\infty)\to[0,\infty)$ with $f(0)=0$ solving
		\begin{equation}\label{eq:f-ode}
			2f''(r)-rK(r)f'(r)+r=0,
			\qquad r>0,
		\end{equation}
		and such that $f'(r)\ge \lambda_0^{-1}$ for all $r>0$ and $f(r)\le f'(0)\,r$.
		One explicit choice is
		\[
			f'(r)=\frac12\int_r^\infty s\exp\Big(-\frac12\int_r^s \tau K(\tau)\,d\tau\Big)\,ds,
		\]
		and the stated properties are proved in \cite[Lemma~5.1]{du2023sequential}.

		Set $\rho_t:=f(R_t)$. By It\^o's formula, using \eqref{eq:dR} and that the quadratic variation of the martingale part is $4r_n(X_t,X'_t)^2dt$, we obtain for $t<\tau_0$,
		\[
			d\rho_t
			\le
			\Big(-R_tK(R_t)f'(R_t)+2r_n(X_t,X'_t)^2f''(R_t)\Big)\,dt
			+2r_n(X_t,X'_t)f'(R_t)\,dW_t.
		\]
		When $R^X_t\ge 2n^{-1}$, we have $r_n(X_t,X'_t)=1$, and then \eqref{eq:f-ode} yields
		\[
			- R_tK(R_t)f'(R_t)+2f''(R_t)=-R_t.
		\]
		Hence on $\{R^X_t\ge 2n^{-1}\}$,
		\begin{equation}\label{eq:drho-large}
			d\rho_t \le -R_t\,dt + dM_t
			\le -\frac{1}{f'(0)}\,\rho_t\,dt + dM_t,
		\end{equation}
		where $M_t$ is a local martingale and we used $f(R_t)\le f'(0)R_t$.

		When $R^X_t\le 2n^{-1}$, we are in the synchronous regime $r_n=0$, and \eqref{eq:dRY} implies that the drift of $R_t$
		is strictly negative up to an $O(n^{-1})$ term; more precisely,
		\[
			\frac{d}{dt}R_t \le -\frac{\lambda}{2}R_t + O(n^{-1}),
			\qquad\text{on }\{R^X_t\le 2n^{-1}\}.
		\]
		Implement this by the concave truncation
		\[
			f_n(r):=
			\begin{cases}
				f(r),                          & r\ge 2n^{-1}, \\[0.4ex]
				\frac{f(2n^{-1})}{2n^{-1}}\,r, & r<2n^{-1}.
			\end{cases}
		\]
		Then $f_n$ is concave and $f_n\uparrow f$. Repeating the It\^o computation for $f_n(R_t)$ yields
		\begin{equation}\label{eq:gronwall-fn}
			\frac{d}{dt}\,\mathbb{E}\big[f_n(R_t)\big]
			\le
			-c\,\mathbb{E}\big[f_n(R_t)\big] + \varepsilon_n,
		\end{equation}
		where $\varepsilon_n\to0$ as $n\to\infty$ and
		\[
			c:=\frac{1}{f'(0)}.
		\]
		Moreover, a direct computation gives
		\[
			f'(0)\le \frac{1}{\lambda_0}+\frac{2M}{\lambda_0^{3/2}}\exp\Big(\frac{M^2}{4\lambda_0}\Big),
		\]
		and therefore
		\[
			c=\Big(\frac{1}{\lambda_0}+\frac{2M}{\lambda_0^{3/2}}\exp\Big(\frac{M^2}{4\lambda_0}\Big)\Big)^{-1}.
		\]

		Applying Grönwall's inequality to \eqref{eq:gronwall-fn} yields
		\[
			\mathbb{E}\big[f_n(R_t)\big]\le e^{-ct}\,\mathbb{E}\big[f_n(R_0)\big]+\frac{\varepsilon_n}{c}.
		\]
		Letting $n\to\infty$, monotone convergence gives
		\[
			\mathbb{E}\big[f(R_t)\big]\le e^{-ct}\,\mathbb{E}\big[f(R_0)\big].
		\]
		Since $f'(r)\ge \lambda_0^{-1}$, we also have $f(r)\ge \lambda_0^{-1}r$, hence
		\[
			\mathbb{E}[R_t]\le \lambda_0\,\mathbb{E}[f(R_t)]
			\le \lambda_0\,e^{-ct}\,\mathbb{E}[f(R_0)]
			\le \lambda_0 f'(0)\,e^{-ct}\,\mathbb{E}[R_0].
		\]
		Recalling $R_t=R^X_t+\frac{2C_L}{\lambda}R^Y_t$ and the definition of $d_\lambda$, this yields
		\[
			W_{d_\lambda}\big(\Law(X_t,Y_t),\Law(X'_t,Y'_t)\big)
			\le C e^{-ct}\,W_{d_\lambda}\big(\Law(X_0,Y_0),\Law(X'_0,Y'_0)\big),
		\]
		with
		\[
			C:=\lambda_0 f'(0)\le 1+\frac{2M}{\sqrt{\lambda_0}}\exp\Big(\frac{M^2}{4\lambda_0}\Big),
		\]
		which is the announced estimate.

		\medskip
		\noindent\emph{Step 4: existence and uniqueness of the invariant measure.}
		The exponential contraction in the complete metric space
		$\big(\mathcal{P}_1(\mathbb{T}^d\times\mathbb{R}^D),W_{d_\lambda}\big)$ implies that the Markov semigroup of $(X_t,Y_t)$ has a unique invariant probability measure $\rho_\lambda$, and that $\Law(X_t,Y_t)$ converges to $\rho_\lambda$ exponentially fast in $W_{d_\lambda}$, uniformly over initial laws with finite first moment.
	\end{proof}

	\section{Proof of Lemma~\ref{lem-gradiantu-bound}}\label{app-lem}

	\begin{proof}
		We follow the Bernstein-type argument of \cite[Lemma~1.5]{CardaliaguetPorretta2019ergodicMaster}
		and adapt it to our Hamiltonian structure. Fix $T>0$ and $y\in\mathbb{R}^D$.
		Let $u=u^{T,y}$ be the classical solution of \eqref{eq-T-hjb}:
		\[
			\partial_t u + \Delta u + H(x,\nabla u) + V(y,x)=0,\qquad u(T,\cdot)=0.
		\]

		\medskip
		\noindent\emph{Step 1. From semiconcavity to a Lipschitz bound on $\mathbb{T}^d$.}
		For $\varphi\in C^2(\mathbb{T}^d)$ define
		\[
			[\varphi]_{\mathrm{sc}}
			:=\sup_{x\in\mathbb{T}^d}\ \sup_{|\xi|=1}\bigl(\partial^2_{\xi\xi}\varphi(x)\bigr)^+.
		\]
		Since $\mathbb{T}^d$ is compact, one has the deterministic estimate (see
		\cite[Lemma~1.3]{CardaliaguetPorretta2019ergodicMaster})
		\begin{equation}\label{eq:sc-to-lip-app}
			\|\nabla \varphi\|_\infty \le \sqrt{d}\,[\varphi]_{\mathrm{sc}}.
		\end{equation}
		Hence it is enough to bound $[u(t,\cdot)]_{\mathrm{sc}}$ uniformly in $t\in[0,T]$.

		\medskip
		\noindent\emph{Step 2. Equation for directional second derivatives.}
		Fix a unit vector $\xi\in\mathbb{R}^d$ and set
		\[
			w(t,x):=\partial^2_{\xi\xi}u(t,x).
		\]
		Differentiate the HJB equation twice in direction $\xi$. Writing $p=\nabla u(t,x)$,
		we obtain
		\begin{align}\label{eq:w-eq-app}
			\partial_t w + \Delta w + H_p(x,p)\cdot\nabla w
			 & \;+\;\big\langle H_{pp}(x,p)\,\nabla u_\xi,\nabla u_\xi\big\rangle
			\;+\;2\big\langle H_{xp}(x,p)\xi,\nabla u_\xi\big\rangle \nonumber    \\
			 & \;+\; H_{xx}(x,p)[\xi,\xi]
			\;+\;\partial^2_{\xi\xi}V(y,x)
			\;=\;0,
		\end{align}
		where $u_\xi=\partial_\xi u$ and $H_{xx}[\xi,\xi]=\xi^\top H_{xx}\xi$.

		Let $(t_*,x_*)$ be a maximum point of $w$ on $[0,T]\times\mathbb{T}^d$.
		If $t_*=T$, then $u(T,\cdot)\equiv 0$ implies $w(T,\cdot)\equiv 0$ and we are done.
		Assume $t_*<T$. At $(t_*,x_*)$ we have $\partial_t w\ge 0$, $\nabla w=0$, $\Delta w\le 0$.
		Evaluating \eqref{eq:w-eq-app} at $(t_*,x_*)$ gives
		\begin{equation}\label{eq:w-max-ineq-app}
			\big\langle H_{pp}(x_*,p_*)\,\nabla u_\xi,\nabla u_\xi\big\rangle
			+2\big\langle H_{xp}(x_*,p_*)\xi,\nabla u_\xi\big\rangle
			+ H_{xx}(x_*,p_*)[\xi,\xi]
			+\partial^2_{\xi\xi}V(y,x_*) \;\ge\;0,
		\end{equation}
		with $p_*=\nabla u(t_*,x_*)$.

		\medskip
		\noindent\emph{Step 3. Structural bounds on $H$ and $V$.}
		By Assumption~\ref{ass-potential} and the bounds on $\ell$ and $\Phi$,
		\begin{equation}\label{eq:Vxx-app}
			\sup_{y\in\mathbb{R}^D}\ \sup_{x\in\mathbb{T}^d}\ \sup_{|\xi|=1}
			\bigl|\partial^2_{\xi\xi}V(y,x)\bigr|
			\le C_\Phi C_\ell.
		\end{equation}

		Next, recall that $H(x,p)=\inf_\alpha\{\alpha\cdot p+c(x,\alpha)\}$ and
		$\alpha^*(x,p)=\argmin_\alpha\{\alpha\cdot p+c(x,\alpha)\}$ so that $H_p(x,p)=\alpha^*(x,p)$.
		The first-order condition is
		\[
			p+c_\alpha(x,\alpha^*(x,p))=0.
		\]
		By Assumption~\ref{ass-cost}(ii), $c_{\alpha\alpha}$ is uniformly positive definite, hence
		\[
			H_{pp}(x,p)=\partial_p\alpha^*(x,p)=-(c_{\alpha\alpha}(x,\alpha^*(x,p)))^{-1}
			\preceq -\frac{1}{\kappa_2^c}I_d,
		\]
		and therefore
		\begin{equation}\label{eq:Hpp-app}
			\big\langle H_{pp}(x,p)\,q,q\big\rangle\le -\frac{1}{\kappa_2^c}|q|^2,\qquad q\in\mathbb{R}^d.
		\end{equation}

		Moreover, Assumption~\ref{ass-cost}(iii) implies polynomial growth bounds on the mixed and second $x$-derivatives.
		More precisely, there exists a constant $C_H>0$, depending only on
		$(\kappa_2^c,\kappa_3^c,\theta_c)$, such that for all $(x,p)$,
		\begin{equation}\label{eq:HxpHxx-app}
			|H_{xp}(x,p)| \le C_H\bigl(1+|\alpha^*(x,p)|\bigr)^{\theta_c},
			\qquad
			|H_{xx}(x,p)| \le C_H\bigl(1+|\alpha^*(x,p)|\bigr)^{1+\theta_c}.
		\end{equation}
		(See \cite[Proof of Lemma~1.5]{CardaliaguetPorretta2019ergodicMaster} for this deduction.)
		Finally, by the strong convexity of $c$ we have the linear growth of the minimizer:
		setting $C_{0,\alpha}:=\sup_{x\in\mathbb{T}^d}|c_\alpha(x,0)|<\infty$, one gets
		\begin{equation}\label{eq:alpha-star-app}
			|\alpha^*(x,p)|\le \kappa_2^c\bigl(|p|+C_{0,\alpha}\bigr),
		\end{equation}
		hence the right-hand sides in \eqref{eq:HxpHxx-app} can be expressed in terms of $|p|$.

		\medskip
		\noindent\emph{Step 4. Closing the bound at the maximum point.}
		Apply Young's inequality to the cross term in \eqref{eq:w-max-ineq-app}:
		for any $\varepsilon>0$,
		\[
			2\langle H_{xp}\xi,\nabla u_\xi\rangle
			\le \varepsilon|\nabla u_\xi|^2 + \varepsilon^{-1}|H_{xp}\xi|^2.
		\]
		Choose $\varepsilon=\frac{1}{2\kappa_2^c}$ and use \eqref{eq:Hpp-app} in \eqref{eq:w-max-ineq-app} to obtain
		\[
			0
			\le
			-\frac{1}{2\kappa_2^c}|\nabla u_\xi|^2
			+ 2\kappa_2^c|H_{xp}(x_*,p_*)|^2
			+ |H_{xx}(x_*,p_*)|
			+ C_\Phi C_\ell.
		\]
		Using \eqref{eq:HxpHxx-app} and \eqref{eq:alpha-star-app}, we infer that
		\begin{equation}\label{eq:nab-uxi-app}
			|\nabla u_\xi(t_*,x_*)|^2
			\le
			C\Bigl(1+|p_*|\Bigr)^{2\theta_c}
			+
			C\Bigl(1+|p_*|\Bigr)^{1+\theta_c}
			+
			C,
		\end{equation}
		for a constant $C$ depending only on
		$(\kappa_0^c,\kappa_2^c,\kappa_3^c,\theta_c,C_\Phi,C_\ell)$.

		Now we relate $|p_*|$ to the maximal positive second derivative.
		Since $w(t_*,x_*)=\max_{t,x}\partial^2_{\xi\xi}u(t,x)$, we have
		\[
			[u(t_*,\cdot)]_{\mathrm{sc}}
			\le w(t_*,x_*),
		\]
		and therefore by \eqref{eq:sc-to-lip-app},
		\[
			|p_*| = |\nabla u(t_*,x_*)|
			\le \|\nabla u(t_*,\cdot)\|_\infty
			\le \sqrt{d}\,[u(t_*,\cdot)]_{\mathrm{sc}}
			\le \sqrt{d}\,w(t_*,x_*).
		\]
		On the other hand, at $(t_*,x_*)$ we also have $|\nabla u_\xi|\ge \partial^2_{\xi\xi}u=w(t_*,x_*)$,
		since $\partial_\xi u_\xi=w$ is one component of the gradient $\nabla u_\xi$.
		Thus \eqref{eq:nab-uxi-app} yields an inequality of the form
		\[
			w(t_*,x_*)^2 \le C\bigl(1+w(t_*,x_*)\bigr)^{1+\theta_c}+C\bigl(1+w(t_*,x_*)\bigr)^{2\theta_c}+C.
		\]
		Because $\theta_c\in(0,1)$, the right-hand side grows strictly subquadratically in $w$,
		so $w(t_*,x_*)$ is bounded by a constant independent of $T$ and $y$.
		Consequently $\sup_{t\in[0,T]}[u(t,\cdot)]_{\mathrm{sc}}<\infty$ uniformly in $(T,y)$, and
		\eqref{eq:sc-to-lip-app} gives the desired uniform estimate
		\[
			\|\nabla_x u^{T,y}_t\|_\infty \le C_{\nabla u},\qquad t\in[0,T],
		\]
		where $C_{\nabla u}$ depends only on $(\kappa_0^c,\kappa_1^c,\kappa_2^c,\kappa_3^c,\theta_c,C_\ell,C_\Phi)$.
	\end{proof}

	\section{Proof of Proposition~\ref{prop:mu-properties}}\label{app-mu}

	\begin{proof}
		We follow the argument of \cite[Proposition~8]{DuRenSuciuWang2024}.
		Recall the definition of the probability measure $\mu$ in \eqref{eq:def-mu}.

		\medskip
		\noindent\emph{Step 1. A stationarity identity.}
		Since $\rho_\lambda$ is invariant for the Markov generator of $(X_t,Y_t)$,
		for every smooth test function $\varphi\in C^2(\mathbb{T}^d\times\mathbb{R}^D)$ with suitable growth we have
		\begin{equation}\label{eq:generator-id}
			\int \mathcal{L}_\lambda \varphi(x,y)\,\rho_\lambda(dx\,dy)=0,
		\end{equation}
		where
		\[
			\mathcal{L}_\lambda \varphi
			=
			\Delta_x \varphi
			+ H_p\bigl(x,\nabla u^{y}(x)\bigr)\cdot \nabla_x\varphi
			+ \lambda\bigl(\ell(x)-y\bigr)\cdot\nabla_y\varphi .
		\]
		In particular, choosing $\varphi(x,y)=\phi(y)$ with $\phi\in C^1(\mathbb{R}^D)$ gives
		\begin{equation}\label{eq:stationary-y}
			\int \nabla\phi(y)\cdot(\ell(x)-y)\,\rho_\lambda(dx\,dy)=0.
		\end{equation}

		\medskip
		\noindent\emph{Step 2. Orthogonality $\int \nabla\Psi(y)\cdot\ell(x)\,(\mu-\rho_\lambda)=0$.}
		Fix $\Psi\in C^2(\mathbb{R}^D)$ with bounded Hessian.
		Define $\phi(y):=\nabla\Psi(y)\cdot y-\Psi(y)$. Then $\nabla\phi(y)=\nabla^2\Psi(y)\,y$.
		Apply \eqref{eq:stationary-y} with this choice of $\phi$:
		\[
			0=\int \bigl(\nabla^2\Psi(y)\,y\bigr)\cdot(\ell(x)-y)\,\rho_\lambda(dx\,dy).
		\]
		Expanding gives
		\begin{equation}\label{eq:rho-identity}
			\int \ell(x)^\top \nabla^2\Psi(y)\,y\ \rho_\lambda(dx\,dy)
			=
			\int y^\top \nabla^2\Psi(y)\,y\ \rho_\lambda(dx\,dy).
		\end{equation}
		On the other hand, by definition of $\mu$ and the relation $y=\langle \ell,m_\lambda\rangle$,
		\begin{multline*}
			\int \ell(x)^\top \nabla^2\Psi(y)\,y\ \mu(dx\,dy)
			=
			\mathbb{E}\!\left[
				\int \ell(x)^\top \nabla^2\Psi(Y_\lambda)\,Y_\lambda\,m_\lambda(dx)
				\right]\\
			=
			\mathbb{E}\!\left[
				Y_\lambda^\top \nabla^2\Psi(Y_\lambda)\,Y_\lambda
				\right]
			=
			\int y^\top \nabla^2\Psi(y)\,y\ \rho_\lambda^Y(dy).
		\end{multline*}
		Since $\mu^Y=\rho_\lambda^Y$, the last term equals the right-hand side of \eqref{eq:rho-identity}.
		Therefore,
		\[
			\int \ell(x)^\top \nabla^2\Psi(y)\,y\ \mu(dx\,dy)
			=
			\int \ell(x)^\top \nabla^2\Psi(y)\,y\ \rho_\lambda(dx\,dy).
		\]
		Finally note that
		\[
			\nabla\Psi(y)\cdot\ell(x)
			=
			\ell(x)^\top \nabla^2\Psi(y)\,y
			\quad\text{up to adding a term depending only on }y,
		\]
		and any term depending only on $y$ integrates the same against $\mu$ and $\rho_\lambda$ because $\mu^Y=\rho_\lambda^Y$.
		This yields the desired orthogonality:
		\[
			\iint \nabla\Psi(y)\cdot\ell(x)\,(\mu-\rho_\lambda)(dx\,dy)=0.
		\]
		Taking $\Psi(y)=y^{(k)}$ gives $\int \ell(x)\,(\mu^X-\rho_\lambda^X)(dx)=0$.

		\medskip
		\noindent\emph{Step 3. Conditional expectation property.}
		Disintegrate $\mu(dx\,dy)=\mu^{X|Y}(dx\mid y)\,\mu^Y(dy)$.
		Fix any bounded measurable $g:\mathbb{R}^D\to\mathbb{R}$ and take $f(x,y)=\ell(x)\,g(y)$ in the definition of $\mu$:
		\begin{multline*}
			\int g(y)\left(\int \ell(x)\,\mu^{X|Y}(dx\mid y)\right)\mu^Y(dy)
			=
			\iint \ell(x)\,g(y)\,\mu(dx\,dy)\\
			=
			\mathbb{E}\!\left[g(Y_\lambda)\int \ell(x)\,m_\lambda(dx)\right]
			=
			\mathbb{E}\!\left[g(Y_\lambda)\,Y_\lambda\right]
			=
			\int g(y)\,y\,\mu^Y(dy).
		\end{multline*}
		By the arbitrariness of $g$ and the fact $\mu^Y=\rho_\lambda^Y$, we obtain
		\[
			\langle \ell,\mu^{X|Y}(\cdot\mid y)\rangle = y
			\qquad\text{for }\rho_\lambda^Y\text{-a.e. }y,
		\]
		which is \eqref{eq:conditional-expectation}.
	\end{proof}

\end{appendix}

\subsubsection*{Funding}
Z.\,R.'s research is supported by the Finance For Energy Market Research Centre,
the France 2030 grant (ANR-21-EXES-0003), and the PEPR PDE-AI project.

\subsubsection*{Acknowledgements}

M.L. is affiliated with the Shanghai Frontiers Science Center of Artificial Intelligence and Deep Learning, and the NYU-ECNU Institute of Mathematical Sciences, at NYU Shanghai.

\bibliographystyle{plain}
\bibliography{ref}

\end{document}